\documentclass[12pt]{amsart}

\usepackage[margin=1in]{geometry}
\usepackage{amsmath}
\usepackage{amssymb}
\usepackage{tikz-cd}

\usepackage{hyperref}
\usepackage{thmtools}
\usepackage[nameinlink, capitalize]{cleveref}
\hypersetup{
    colorlinks=true,
    linkcolor=black,
    filecolor=green,      
    urlcolor=cyan,
    citecolor=magenta,
    pdftitle={The Mumford--Tate group of the very general Lagrangian fiber},
    pdfauthor={Edward Varvak},
}
\usepackage{caption} 
\newcommand{\R}{\mathbb{R}}
\newcommand{\Q}{\mathbb{Q}}
\newcommand{\C}{\mathbb{C}}
\newcommand{\Z}{\mathbb{Z}}
\newcommand{\Aut}{\mathrm{Aut}}
\newcommand{\Hom}{\mathrm{Hom}}
\newcommand{\id}{\mathrm{id}}

\newcommand{\bbH}{\mathbb{H}}
\newcommand{\bbP}{\mathbb{P}}

\newcommand{\cT}{\mathcal{T}}
\newcommand{\cV}{\mathcal{V}}
\newcommand{\cP}{\mathcal{P}}
\newcommand{\cA}{\mathcal{A}}

\newcommand{\Sp}{\mathrm{Sp}}
\newcommand{\SL}{\mathrm{SL}}
\newcommand{\GL}{\mathrm{GL}}
\newcommand{\ol}{\overline}
\newcommand{\ul}{\underline}
\newcommand{\wt}{\widetilde}
\newcommand{\wh}{\widehat}

\newcommand{\rightQ}[2]{\left.\raisebox{.2em}{$#1$}\middle/\raisebox{-.2em}{$#2$}\right.}
\newcommand{\leftQ}[2]{\left.\raisebox{-.2em}{$#2$}\middle\backslash\raisebox{.2em}{$#1$}\right.}

\DeclareMathOperator{\im}{\mathrm{im}}

\newtheorem{theorem}{Theorem}[section]
\newtheorem{proposition}[theorem]{Proposition}
\newtheorem{lemma}[theorem]{Lemma}
\newtheorem{corollary}[theorem]{Corollary}
\newtheorem{question}[theorem]{Question}

\theoremstyle{definition}
\newtheorem{definition}[theorem]{Definition}
\newtheorem{example}[theorem]{Example}
\newtheorem{remark}[theorem]{Remark}

\numberwithin{equation}{section}

\title{The Mumford--Tate group of a very general Lagrangian fiber}
\author{Edward Varvak}
\address{\parbox{\linewidth}{
Department of Mathematics \\
University of Chicago \\
5734 S University Ave \\
Chicago, IL 60637 \\}}

\email{varvak@uchicago.edu}

\begin{document}

\begin{abstract}
    For a primitive symplectic variety \(X\) admitting a holomorphic Lagrangian fibration 
    \(f: X \to B\), we classify the possible Mumford--Tate groups of the very 
    general fiber. We extract as a consequence that the discriminant of 
    \(f\) has codimension one in \(B\). In the special case when \(X\) is a hyper-K\"ahler manifold which 
    is general among those admitting a Lagrangian fibration, and \(b_2(X) \geq 7\), 
    we further show that the very general fiber of \(f\) is Hodge-generic. Finally, as an 
    application of our classification, we confirm an expectation of Li--Tosatti that 
    when \(f\) is not isotrivial, the special K\"ahler metric on \(B\) has positive 
    holomorphic sectional curvature on the regular values of \(f\).
\end{abstract}

\maketitle

\section{Introduction}

In this paper, we will use Hodge-theoretic methods to study holomorphic symplectic 
varieties. We will start with the following definitions.

\begin{definition}
    Let \(X\) and \(B\) be normal analytic varieties, and \(f: X \to B\) a non-trivial 
    surjective morphism with connected fibers.
    \begin{enumerate}
        \item We say that \(X\) is {\bf holomorphic symplectic} if it admits a 
        non-degenerate 2-form \(\sigma \in H^0(X^\mathrm{sm}, \Omega_{X^\mathrm{sm}}^2)\) 
        which extends holomorphically onto a resolution of singularities \(Y \to X\).
        \item We say that a subvariety \(Z \subseteq X\) is {\bf Lagrangian} if 
        \(\sigma|_Z = 0\) and \(\dim Z = \frac{1}{2} \dim X\). We say that 
        \(f: X \to B\) is a Lagrangian fibration if every fiber of \(f\) is Lagrangian.
        \item Suppose that \(X\) is compact, \(\Q\)-factorial, and log-terminal. 
        We say that \(X\) is {\bf primitive symplectic} if \(H^1(X, \mathcal{O}_X) = 0\) 
        and the symplectic form \(\sigma\) is unique up to scaling.
        \item If \(X\) is a compact, simply connected K\"ahler manifold with a symplectic 
        form \(\sigma\) unique up to scaling, we say that \(X\) is {\bf hyper-K\"ahler}.
    \end{enumerate}
\end{definition}

By the Beauville--Bogomolov decomposition theorem, any smooth compact symplectic variety
is isogenous to a product of hyper-K\"ahler manifolds and complex tori. Their work has 
more recently been generalized to compact symplectic varieties, which are isogenous to 
products of complex tori and special classes of primitive symplectic varieties (see, 
for instance, \cite[Theorem 1.5]{horing-peternell19}).

When \(X\) is hyper-K\"ahler, much is understood. For instance, any non-trivial surjective
morphism \(f: X \to B\) with connected fibers is forced to be a Lagrangian fibration \cite{greb-lehn14}, 
and \(B\) is smooth if and only if \(B \simeq \bbP^n\) \cite{hwang08}. 
A conjecture of Matsushita predicts that this is always the case.

More generally, suppose \(X\) is any primitive symplectic variety admitting a 
Lagrangian fibration \(f: X \to B\). We abuse notation to denote \(f: X^\circ \to B^\circ\) 
the restriction of \(f\) to the preimage of the regular values, over which the fibers 
of \(f\) are smooth. It turns out that every smooth fiber \(X_b\) is an Abelian 
variety, due to Matsushita \cite{matsushita99} when \(X\) is projective hyper-K\"ahler 
and \cite{campana06} in more general settings.

Take \(X\) to have dimension \(2g\), so that both the base \(B\) and the fibers \(X_b\)
are \(g\)-dimensional. We denote by \(\cP: B^\circ \to \cA_{g, \eta}\) the period map 
associated to \(f\), where \(\cA_{g, \eta}\) is a moduli space for appropriately-polarized
Abelian varieties. If \(\cP\) is non-constant, that is, if \(f\) is not an isotrivial fibration,
then \(\cP\) is generically injective; this is due to \cite{geemen-voisin16} for 
\(X\) a very general hyper-K\"ahler and \cite{bakker25} for primitive symplectic \(X\).
In the non-isotrivial case, we say that \(f\) varies maximally.

Our main technical result provides some description of the image of \(\cP\). Since 
\(\cA_{g, \eta}\) is a Shimura variety; explicitly,
\begin{equation*}
    \cA_{g, \eta} = \leftQ{\rightQ{\Sp_{2g}(\R)}{U(g)}}{\Gamma},
\end{equation*}
some ``special'' subvarieties of \(\cA_{g, \eta}\) are induced by algebraic 
subgroups \(G < \Sp_{2g}\). This subgroup is the Mumford--Tate group of the
very general fiber of \(f\). Our result is a full classification of Mumford--Tate 
groups which can occur in a Lagrangian fibration's very general fiber. 

\begin{theorem}[=\Cref{mt-classification}] 
    Let \(X\) be a primitive symplectic variety of dimension \(2g\), and let 
    \(f: X \to B\) be a Lagrangian fibration. The Mumford--Tate group of 
    \(f\) is isogenous to one of 
    \begin{equation*}
        \mathbf{MT}(X_b) \sim \begin{cases}
            (\mathbb{G}_m \times \Sp_{2g/k}^k) / \{\pm 1\} & : f \text{ varies maximally, } 
            k \text{ divides g,}\\
            \mathrm{Res}_{K/\Q} \mathbb{G}_{m, K} & : 
            f \text{ is isotrivial with fiber isogenous to } E^g \text{ for } \\
            & \quad  E \text{ a CM elliptic curve with endomorphism field } K\\
            \GL_2 & : \text{ otherwise.}
        \end{cases}
    \end{equation*}
\end{theorem}

There is a more geometric description of this result, which we provide as an 
obvious corollary. In particular, this generalizes \cite{matsushita99} and 
\cite{campana06} to describe the structure of the fibers. 

\begin{corollary}
    Suppose that \(f: X \to B\) is non-isotrivial. Then for some positive integer 
    \(k\) dividing \(g\), the very general fiber \(X_b\) is isogenous to a product 
    \(A_1 \times \cdots \times A_k\), where the \(A_i\) are Hodge-generic Abelian 
    varieties of dimension \(g/k\), which are not pairwise isogenous. 
    In particular, each \(A_i\) is simple and has no extra endomorphisms.
\end{corollary}

As a direct consequence of this classification, we answer a question of 
Li--Tosatti. Their result \cite[Corollary 2.2]{li-tosatti24} shows that 
if \(f\) is not isotrivial, the special K\"ahler metric \(\omega_\mathrm{SK}\) on 
\(B^\circ\) has positive Ricci curvature. In a later remark, they state their stronger
expectation that \(\omega_\mathrm{SK}\) should have positive holomorphic sectional 
curvature on \(B^\circ\). The following result proves them right.

\begin{proposition}[=\Cref{constant-diagonal-entry}]
    Let \(X\) be primitive symplectic, and \(f: X \to B\) be a Lagrangian fibration 
    with period map \(\cP: B^\circ \to \cA_{g, \eta}\) and period matrix \(\mathbf{T}\), written as
    \begin{equation*}
        \cP(b) = \begin{pmatrix}
            \mathbf{T}(b) \\ I_g
        \end{pmatrix},\qquad \mathbf{T} = \begin{pmatrix}
            \tau_{11} & \cdots & \tau_{1g} \\ 
            \vdots & \ddots & \vdots \\ 
            \tau_{1g} & \cdots & \tau_{gg}
        \end{pmatrix}
    \end{equation*}
    Suppose that \(\mathbf{T}\) has a constant diagonal entry \(\tau_{ii}\)
    in some analytic neighborhood \(U \subseteq B^\circ\). Then \(f\) is isotrivial.
\end{proposition}

By \cite[Remark 2.3]{li-tosatti24}, \(\omega_\mathrm{SK}\) having positive holomorphic 
sectional curvature is equivalent to \(\mathbf{T}\) having non-constant diagonal
entries, as needed. Furthermore, we will see that the proof of \Cref{constant-diagonal-entry} 
tells us a fair bit about the existence of any constant entry in the period matrix.

\begin{corollary}[=\Cref{omega_SK}]
    Let \(X\) be primitive symplectic, let \(f: X \to B\) be a Lagrangian fibration, 
    and suppose that \(f\) is not isotrivial. Then \(\omega_\mathrm{SK}\) has positive 
    holomorphic sectional curvature. If we suppose further that 
    \(\mathbf{MT}(X_b) = \mathrm{GSp}_{2g}\), then \(\omega_\mathrm{SK}\) has 
    positive bisectional curvature. 
\end{corollary}

It turns out that the deformation space of a  polarized hyper-K\"ahler manifold \((X, H)\)
is a Shimura variety. We can construct it as follows: take 
\(L = H^\perp \subsetneq H^2(X, \Z)\) with respect to the Beauville--Bogomolov--Fujiki 
form \(q\), and define 
\[\mathcal{D}_H = \{[\sigma] \in \bbP L_\C \mid q(\sigma, \sigma) = 0, \quad
q(\sigma, \ol\sigma) > 0\}.\]
Then the deformation space of \(X\) is constructed as an open subset 
\begin{equation*}
    \mathcal{X} \subseteq \leftQ{\mathcal{D}_H}{\Gamma} = 
    \leftQ{\rightQ{\mathrm{SO}(2, b_2-3)}{\mathrm{SO}(2) \times \mathrm{SO}(b_2-3)}}{\Gamma},
\end{equation*}
where \(b_2 = \dim H^2(X, \C)\).

As such, to construct a deformation space of hyper-K\"ahler manifold, it suffices 
to exhibit one member. There are currently four known deformation types: the infinite 
families \(K3^{[g]}\) and \(\mathrm{Kum}_g\), occurring in every even dimension \(2g\),
and the two sporadic types \(OG6\) and \(OG10\), occurring in dimensions 6 and 10 
respectively. 

Every known deformation space \(\mathcal{X}\) contains a member \(X\) admitting a 
Lagrangian fibration \(f: X \to B\). Furthermore, the hyper-K\"ahler SYZ conjecture 
predicts that such a member exists in every deformation class of hyper-K\"ahler 
manifolds. Matsushita subsequently shows that a codimension one collection of 
deformations \(X'\) of \(X\) preserve the Lagrangian fibration \(f\) \cite{matsushita16}.

Thus, we have a notion of \(X\) being general among the members of its deformation 
class admitting a Lagrangian fibration. Our next result strengthens 
\Cref{mt-classification} in this setting, concluding that the very general fiber 
\(X_b\) is Hodge-generic.

\begin{theorem}[=\Cref{general-mt-classification}]
    Let \(\mathcal{X}\) be a deformation class of hyper-K\"ahler manifolds satisfying 
    \begin{enumerate}
        \item there exists \(X \in \mathcal{X}\) a hyper-K\"ahler manifold admitting a 
        Lagrangian fibration,
        \item the second Betti number of any member \(X \in \mathcal{X}\) is at least 7.
    \end{enumerate}
    Choose \(X \in \mathcal{X}\) general among those admitting Lagrangian fibrations,
    \(f: X \to B\). Then the Mumford--Tate group of the very general fiber \(X_b\) 
    is \(\mathrm{GSp}_{2g}\).
\end{theorem}

Once again, it is an obvious consequence that the very general fiber \(X_b\) is 
simple with no extra endomorphisms.

We prove these results by analyzing the monodromy of the variation of Hodge structures 
underlying \(f\). In the process of proving \Cref{mt-classification}, we also 
produce the following consequence.

\begin{theorem}[=\Cref{codim-one-discriminant}]
    Let \(X\) be a primitive symplectic variety, and \(f: X \to B\) a Lagrangian 
    fibration. Let \(D \subseteq B\) be the discriminant of \(f\). Then \(D\) has 
    codimension one.
\end{theorem}

To the best of our knowledge, this was not previously known for primitive symplectic
varieties. Hwang \cite[Proposition 4.1]{hwang08} proved this result when \(X\) 
is a projective hyper-K\"ahler, but his argument does not actually need \(X\) to be 
projective. However, it relies on \(B\) being simply connected of Picard rank 1, and 
therefore fails for primitive symplectic varieties.

\subsection*{Proof outline}

We use the same notation as above: let \(f: X \to B\) be a Lagrangian fibration, 
\(X\) a primitive symplectic variety of dimension \(2g\), \(X_b\) a very general fiber.
Let \(D\) be the discriminant of \(f\), let \(B^\circ = B \smallsetminus D\).
Let \(V = R^1 f_* \Z_{X^\circ}\) be the variation of Hodge structures underlying the 
smooth fibers of \(f\), and let \(\cP: B^\circ \to \cA_{g, \eta}\) be its 
associated period map.

First, we outline the proof of \Cref{mt-classification}. The idea is as follows: 
our prior results in \cite{varvak25} handle the isotrivial cases, so we may 
assume that \(f\) is not isotrivial. By \cite{bakker25}, it varies maximally. 
Thus the image of \(\cP\) is \(g\)-dimensional in \(\cA_{g, \eta}\). Understanding 
the Mumford--Tate group of \(X_b\) amounts to understanding the image of \(\cP\).

Next we want to show that \(\cP(B^\circ)\) is Hodge-generic inside some product locus 
\(\cA_{h_1, \eta_1} \times \cdots \times \cA_{h_k, \eta_k} \subseteq \cA_{g, \eta}\).
In other words, we need to rule out \(\cP\) from landing in the more exotic special 
subvarieties of \(\cA_{g, \eta}\), such as the loci of real-multiplication 
Abelian varieties. To best describe our argument, recall that the period map 
extends to the Baily--Borel compactification of \(\cA_{g, \eta}\): 
\(\ol{\cP}: B \dashrightarrow \ol{\cA}_{g, \eta}\).

The boundary of \(\ol{\cA}_{g, \eta}\), consisting of degenerations of Abelian varieties, 
has a natural stratification by the formal logarithm rank of the local monodromy operator 
corresponding to the degeneration. In \Cref{infinite-monodromy}, we show that 
there is a component \(D_1 \subseteq D\) and a Zariski open subset 
\(D_1^\circ \subseteq D\) such that \(\ol\cP\) is defined on \(D_1^\circ\).
With \Cref{local-monodromy-eigenvalues}, we see that the \(\ol{\cP}\) sends \(D_1^\circ\) 
to the rank one stratum of the boundary. 

Now the argument is straightforward. We view the \(\Q\)-variation of Hodge structures 
\(V_\Q\) as a representation of the Mumford--Tate group \(\mathbf{MT}(X_b)\).
We need to rule out a number of representations. Suppose that \(V_\Q\) 
were irreducible but \(V_\C = V_\Q \otimes_\Q \C\) split as 
\(U_1 \oplus \cdots \oplus U_k\). Then the splitting must occur over some 
number field \(K/\Q\). Since the \(U_i\) comprise a single Galois orbit, 
they all degenerate simultaneously. In particular, this would force the fibers over 
\(D_1^\circ\) to be degenerations of rank at least \(k\). This forces \(k = 1\).

Thus any splitting of \(V_\C\) must occur over \(\Q\). In particular, this rules 
out the loci of Abelian varieties with special endomorphisms, such as 
the real-multiplication loci of \(\cA_{g, \eta}\). As such, we may prove results 
about \(V_\Q\) by working with \(V_\C\). By the classification of semisimple Lie groups, 
the complex derived Mumford--Tate group 
\([\mathbf{MT}(X_b)(\C), \mathbf{MT}(X_b)(\C)]\) is isogenous to a product of some 
\(\SL_{2h}(\C)\), \(\Sp_{2h}(\C)\), and \(\mathrm{SO}_{2h}(\C)\). We claim 
that they cannot act on \(V_\C\) by any tensor representations. Indeed, the argument 
is analogous to the one above: if \(W_1\) was a weight one variation which degenerated 
with rank one, \(W_1 \otimes W_2\) would degenerate with rank \(\dim W_2\).

Hence we must have that all but one isogeny factor of \(\mathbf{MT}(X_b)\) act 
trivially on \(V_\Q = U_1 \oplus \cdots \oplus U_k\). Now because the 
\(U_i\) correspond to families of Abelian varieties of dimension 
\(\frac{1}{2}\dim U_i\), we can rule out \(\Q\)-forms of 
\(\mathrm{SO}_{2h}\) and \(\mathrm{SL}_{2h}\) from occuring as isogeny factors 
of \(\mathbf{MT}(X_b)\). We therefore only have to worry about 
\(\mathbf{MT}(X_b) \sim \mathbb{G}_m \cdot (\Sp_{2h_1} \times \cdots \times \Sp_{2h_k}\)), 
as we had wanted to show. 

Finally, we need to claim that \(h_1 = \cdots = h_k\). By prior work \cite{varvak25}, 
we know that \(V_\C\) is irreducible as a monodromy representation. By Andr\'e's theorem, 
we can find a collection of automorphisms of \(V_\C\) permuting the irreducible factors 
\(U_i\) transitively. Hence if the \(U_i\) had different numerics, we could 
appeal to the fact that \(V_\Q\) varies maximally to rule them out, as needed.

The proof of the intermediate \Cref{infinite-monodromy} relies on work of 
\cite{grushevsky-mondello-manni-tsimerman25}. We adopt their style of argument: 
suppose that \(B^\circ\) lands inside of a compact subvariety of \(\cA_{g, \eta}\). 
(This is equivalent to \(\cP(B^\circ)\) not meeting the Baily--Borel boundary 
of \(\cA_{g, \eta}\).) Since the image is \(g\)-dimensional, \cite{grushevsky-mondello-manni-tsimerman25} 
show that it must land inside of a compact special subvariety of \(\cA_{g, \eta}\). 
We use the classification result of Satake, listed in \Cref{table-of-algebraic-forms}, 
to rule out every compact special subvarieties from containing the period image. 

Finally, the proof of \Cref{general-mt-classification} follows from a result 
of \cite{lehn2016}, which states that a general singular fiber of type 
\(\mathrm{I}_0^*\) cannot occur in a general deformation of a Lagrangian fibration 
\(f: X \to B\). We use the local mondromy operator permuting the irreducible factors 
of \(V_\Q = U_1 \oplus \cdots \oplus U_k\) to show that \(f\) has fibers of type 
\(\mathrm{I}_0^*\), ruling it out from being general.

Finally, we outline the structure of the paper. 
In \S2, we review preliminaries on variations of weight one Hodge structures, namely
Deligne's canonical extension, Mumford--Tate groups, and standard examples of Lagrangian 
fibrations. We also show that the groups in \Cref{mt-classification} are attainable. 
In \S3, we prove a crucial technical remark of \cite{bakker-schnell23} about how 
Lagrangian fibrations degenerate in codimension one. Then \S4 proves the classification 
\Cref{mt-classification}, as well as the discriminant result
\Cref{codim-one-discriminant}, and the application 
to Li--Tosatti, \Cref{constant-diagonal-entry}. Finally, \S5 studies the general 
deformation of a Lagrangian fibration and proves \Cref{general-mt-classification}.

\subsection*{AI useage}
This paper contains no AI output, but the author used Gemini 3.1 Pro for help with 
research, definitions, and review.

\subsection*{Acknowledgements}
I thank my advisor, Ben Bakker, for his support and guidance in writing this paper.
I'm also grateful to \.{I}zzet Co\c{s}kun, Daniel Litt, and Valentino Tosatti for 
helpful discussions.

\section{Preliminaries}

\subsection{Deligne's canonical extension}

Let \(f: X \to B\) be a family of Abelian varieties, \(D \subseteq B\) the discriminant 
locus of \(f\), \(B^\circ = B \smallsetminus D\), and \(f: X^\circ \to B^\circ\) the 
smooth part of \(f\), as above. We may consider the underlying \(\Z\)-variation of Hodge 
structures \(V = R^1 f_* \Z_{X^\circ}\). Let \(\cV = V \otimes_{\Z} \mathcal{O}_{B^\circ}\) 
be the underlying flat bundle, equipped with the flat connection 
\(\nabla: \cV \to \cV \otimes \Omega_{B^\circ}^1\).

Extending \(\cV\) to all of \(B\) as a flat bundle is impractical, since the VHS 
degenerates along \(D\). However, Deligne constructs a vector bundle \(\ol{\cV}\) 
extending \(\cV\) to \(D^\mathrm{ls}\) the log-smooth part of \(D\), 
along with a logarithmic connection 
\begin{equation*}
    \ol{\nabla}: \ol{\cV} \to \ol{\cV} \otimes \Omega_{B^\mathrm{ls}}^1(\log D^\mathrm{ls})
\end{equation*}
extending \(\nabla\). 
We will outline the local construction of \(\ol{\cV}\) on the smooth part of the 
discriminant \(D^\mathrm{sm}\); this will be enough for the purposes of this paper,
but the general case of \(D^\mathrm{ls}\) is completely analogous, albeit with more
variables. 

Let \(U \simeq \Delta^{g} \subseteq B^\circ\) be a polydisk which only meets 
\(D^\mathrm{sm} \subseteq D\), and let \(U^\circ \simeq \Delta^* \times \Delta^{g-1}\) be 
the complement of the discriminant in \(U\). Denote the restriction of \(\cV\) to 
\(U^\circ\) by \(\cV_U\). The monodromy group of the restricted VHS is generated 
by a single element, which we will denote by \(T \in \Aut(V_b)\) for some 
\(b \in U^\circ\). Let \(t_1, t_2, \dots, t_g\) be the natural coordinates
on \(U\), so that \(U^\circ\) is defined by the non-vanishing of \(t_1\). 

By the classical Monodromy Theorem, \(T\) is quasi-unipotent, so that we can write 
\(T = T_s \cdot T_u\), where \(T_s^m = I_{2g}\) and \(T_u\) is unipotent. In particular,
\(T\) admits a formal logarithm 
\begin{equation*}
    N = \log T = \log T_s + \log T_u = N_s + N_u,
\end{equation*}
where \(N_u\) is nilpotent, uniquely defined by 
\begin{equation*}
    N_u = (T_u - I_{2g}) - \frac{1}{2} (T_u - I_{2g})^2 + \cdots
\end{equation*}
and \(N_s\) is a matrix satisfying the property 
\begin{equation*}
    \exp(-2\pi i N_s) = T_s.
\end{equation*}

Since \(N_s\) is produced by diagonalizing \(T_s\) and applying 
\(z \mapsto \frac{\log z}{2\pi i}\), it depends only on a choice of logarithm branch cut.
As such, \(N_s\) is unique upon choosing a half-open unit interval \(I \subseteq \R\) 
in which the eigenvalues of \(N_s\) lie. We construct sections for the bundle 
\(\ol{\cV}_U\) as follows.

Let \(v = v(t_1, t_2, \dots, t_g) \in V_\C|_U \subseteq \cV_U\) be a flat multi-section of
\(\cV_U\), that is, a multi-section of the underlying \(\C\)-local system \(V_\C|_U\). We allow it
to be multivalued because moving around \(D\) sends \(v\) to \(T.v\). Define 
\begin{equation*}
    \wt{v}(t_1, t_2, \dots, t_g) =
    \exp\left(\frac{\log(t_1) N}{2\pi i}\right).v(t_1, \dots, t_g).
\end{equation*}
It is not hard to see that this new section is single-valued on \(U^\circ\), and 
one readily checks that \(\wt{v}\) has a removable singularity along \(D\). For a 
detailed overview, refer to \cite{cattani16}.

Finally, for a frame of flat multisections \(v_1, \dots, v_r \in \cV_U\), we define
the canonical extension \(\ol{\cV}_U = \bigoplus_{i=1}^r \mathcal{O}_U \wt{v}_i\).
It comes with an induced logarithmic connection 
\begin{equation*}
    \nabla: \wt{v}_j \mapsto \nabla \left(\exp\left(\frac{\log(t_1) N}{2\pi i}\right). v_j\right)
    = \frac{N}{2\pi i} \wt{v}_j \otimes \frac{dt_1}{t_1}.
\end{equation*}
It is a highly non-trivial, but nevertheless true fact that this construction 
globalizes to all of \(B^\mathrm{ls}\); that is, one glues \(\ol{\cV}_U\) along all the
analytic neighborhoods \(U\) of \(D^\mathrm{ls}\) the log-smooth locus of the 
discriminant to canonically produce an extension \(\ol{\cV}\) of \(\cV\).

Next, we turn our attention to the special case where \(V\) is a variation of Hodge 
structures coming from a Lagrangian fibration \(f: X \to B\). In this setting, we can describe 
Deligne's canonical extension in another way. Generalizing a result of Matsushita, 
Schnell \cite[\S 4]{schnell23} proves that the natural isomorphism 
\begin{equation*}
    \iota_{\sigma}: \cT_{B^\circ} \xrightarrow{\sim} F^1 \cV, \quad 
    w \mapsto \sigma(w, -)
\end{equation*}
extends to an isomorphism 
\begin{equation*}
    \iota_{\sigma}: \cT_{B^\mathrm{ls}} \xrightarrow{\sim} F^1 \ol{\cV},
\end{equation*}
where \(F^1 \ol{\cV} := \ol{\cV} \cap j_*(F^1 \cV)\), and \(\ol{\cV}\) is 
Deligne's canonical extension for which \(N_s\) has eigenvalues 
in \((-1, 0]\). Since the locus of \(D\) which fails to be log-smooth 
has codimension at least 2 in \(B\), we may further reflexivize both sheaves 
to produce an isomorphism \(\cT_B \simeq (F^1 \ol{\cV})^{\vee \vee}\); in 
abuse of notation, we will often drop the reflexivization on the right.

\subsection{Mumford--Tate groups of weight one Hodge structures}

Recall that the Deligne torus is the real algebraic group 
\(\mathbb{S} = \mathrm{Res}_{\C/\R} \mathbb{G}_{m, \C}\); the data of a 
weight one \(\Q\)-Hodge structure is the data of a vector space 
\(H_\Q\) together with a representation \(h: \mathbb{S} \to \GL(H_\R)\) 
such that the underlying weight representation 
\begin{equation*}
    \mathbb{G}_{m, \R} \xrightarrow{w} \mathbb{S} \xrightarrow{h} \GL(H_\R)
\end{equation*}
is defined over \(\Q\).

\begin{definition}
    The {\bf Mumford--Tate group} of a \(\Q\)-Hodge structure \(H_\Q\) is the 
    smallest algebraic subgroup \(G \subseteq \GL(H_\Q)\) whose real points
    contain the image of \(h\).
\end{definition}

For a brief introduction to Mumford--Tate groups and Mumford--Tate domains, we 
recommend \cite[Chapter 15]{CMP}. However, we will summarize the necessary basic 
results below.

If \(V\) is a \(\Q\)-VHS on a quasi-projective variety \(B^\circ\), there is 
a countable collection \(\{Z_j\}\) of closed subvarieties of \(B^\circ\) such that 
for any \(b, b' \in B^\circ \smallsetminus \bigcup_{j=1}^\infty Z_j\) 
we have \(\mathbf{MT}(V_b) = \mathbf{MT}(V_{b'})\). For this reason, we will abuse 
notation and write \(\mathbf{MT}(V)\) to denote the Mumford--Tate group of the 
Hodge structure \(V_b\) over the very general point \(b \in B^\circ\).

We note that if \(Z \subseteq B^\circ\) is a closed subvariety, then 
\(\mathbf{MT}(V|_Z) \leq \mathbf{MT}(V)\). A weight one polarized VHS admits a period map 
to a period domain \(\cP: B^\circ \to \cA_{g, \eta}\). Since the Mumford--Tate 
group only shrinks under specialization, we can make the following definition.

\begin{definition}
    Let \(H\) be a weight one polarized \(\Q\)-Hodge structure. We say that \(H\) is 
    {\bf Hodge-generic} if \(H\) has a maximal Mumford--Tate group in \(\cA_{g, \eta}\),
    that is, \(\mathbf{MT}(H) = \mathrm{GSp_{2g, \Q}}\).
\end{definition}

Notice that the VHS \(V\) naturally gives rise to two representations:
the Mumford--Tate group \(\mathbf{MT}(V_b)\) and the fundamental group \(\pi_1(B^\circ, b)\)
both act on the very general fiber \(V_b\). One can ask if these two representations 
are related. A theorem of Yves Andr\'e provides an affirmative answer.

\begin{theorem}[Andr\'e]
    Let \(V\) be a polarized \(\Q\)-VHS on a quasi-projective variety \(B^\circ\), 
    and let \(\rho: \pi_1(B^\circ, b) \to \GL(V_b)\) be the monodromy action on a 
    very general fiber. Take \(M\) to be the Zariski closure of the monodromy group 
    \(\im(\rho) \leq \GL(V_b)\), and let \(M^\circ \leq M\) be its identity component.
    Then \[M^\circ \trianglelefteq [\mathbf{MT}(V_b), \mathbf{MT}(V_b)].\]
\end{theorem}

On the other hand, the group \(\mathbf{MT}(V_b)\) admits an action on \(\cA_{g, \eta}\). 
We refer to the orbit of \(V_b\) under this action as a Mumford--Tate subdomain. It 
can be shown that this orbit is a Hermitian locally symmetric submanifold. Consequently, 
the groups which can occur as \(\mathbf{MT}(V_b)\) are highly
constrained. Due to work of Satake \cite{satake67}, we have a full classification of 
the group \(\mathbf{MT}(V_b)\) and its action on \(V_b\), up to a central isogeny. 
These groups are presented in \Cref{table-of-real-forms}, along with some useful associated
invariants.

\begin{table}
    \centering
    \begin{tabular}{|c||c|c|}
        \hline && \\
        {\bf Real group } & {\bf Dimension of } & {\bf Dimension of } \\
        {\bf (up to isogeny)} & {\bf Mumford--Tate subdomain} & {\bf representation}\\ && \\
        \hline && \\
        \(\SL_2(\R)\) & 1 & 2 \\ && \\ \hline &&\\
        \(\mathrm{SU}(p, n-p)\) & \(p(n-p)\) & \(n\) \\ && \\ \hline &&\\
        \(\mathrm{SU}(1, n-1)\) & \(n-1\) & \(\binom{n}{c}\)\\ && \(c \in \{2, 3, \dots, n-2\}\) \\ && \\ \hline &&\\
        \(\mathrm{SO}^*(2r)\) & \(\binom{r}{2}\) & \(2r\) \\ && \\ \hline &&\\
        \(\Sp(r)\) & \(\binom{r+1}{2}\) & \(2r\) \\ && \\ \hline &&\\
        \(\mathrm{SO}(2p-2, 2)\) & \(2p-2\) & \(2^{p-1}\) \\ && \\ \hline &&\\
        \(\mathrm{SO}(2p-1, 2)\) & \(2p-1\) & \(2^p\) \\ && \\  
        \hline
    \end{tabular} \\
    \caption{Satake's classification of real forms of irreducible weight one Mumford--Tate representations.}
    \label{table-of-real-forms}
\end{table}

\subsection{Some examples of Lagrangian fibrations and their Mumford--Tate groups}

To place our results in context, we will consider some standard examples of Lagrangian
fibrations, and figure out the Mumford--Tate groups of their very general fibers. 

\begin{example}
Let \(S\) be an elliptically fibered K3 surface, and let \(f: S \to \bbP^1\) be the 
associated fibration. Then \(f\) is a Lagrangian fibration and the fibers of \(f\) 
are genus one curves. The image of the associated period map has dimension either 
one or zero.

In the former case, the period map dominates \(\cA_1\), so the Mumford--Tate group 
of the very general fiber \(\mathbf{MT}(S_b) = \GL_2\). In the latter case, 
every fiber of \(f\) is isogenous to some elliptic curve \(E\), so that 
\(\mathbf{MT}(S_b) = \mathbf{MT}(E)\) is one of \(\GL_2\) or 
\(\mathrm{Res}_{K/\Q} \mathbb{G}_{m, K}\), depending on whether \(E\) has CM. 
Here \(K\) is the CM field of \(E\). \qed
\end{example}

From this trivial example, it is clear that our \Cref{mt-classification} holds when 
\(X_b\) has dimension one. Put differently, the only special subvarieties of \(\cA_1\) 
are \(\cA_1\) and CM points. However, in higher dimensions, we see from 
\Cref{table-of-real-forms} that more exotic special subvarieties start to appear. 
Our claim is that despite their existence, Lagrangian fibrations tend to avoid them, 
behaving like the one-dimensional case. 

Next, we consider an immediate generalization of the preceding example.

\begin{example}[The \(K3^{[n]}\) fibration] \label{example-K3n}
As before, take \(S\) to be an elliptically fibered K3 surface, and \(\pi: S \to \bbP^1\) 
the associated Lagrangian fibration. Recall that for a smooth surface \(S\), 
the Hilbert scheme of \(g\) points \(X := S^{[g]}\) is also smooth. Note that 
the map \(g\) induces a proper morphism \(f := \pi^{[g]}: X \to (\bbP^1)^{[g]} = \bbP^g\).
One can readily see that \(X\) is hyper-K\"ahler and \(f\) is a Lagrangian fibration.

Notice that the general fiber of \(f\) has the form 
\(X_{(b_1, \dots, b_g)} = S_{b_1} \times \cdots \times S_{b_g}\). It is a product of 
distinct genus one curves, so it lies in the product locus 
\((\cA_1)^g \subseteq \cA_{g, \eta}\). Again, we have three cases to consider. 
If the period map associated to \(\pi: S \to \bbP^1\) had one-dimensional image, 
then the period map associated to \(f: X \to \bbP^g\) has \(g\)-dimensional image. 
In this case, it is forced to dominate \((\cA_1)^g\), so that the Mumford--Tate group of 
the very general fiber \(\mathbf{MT}(X_b) = (\mathbb{G}_m \times (\SL_2)^g) / \pm 1\).

On the other hand, if \(\pi\) was isotrivial with fiber \(E\), then \(f\) is also 
isotrivial, and there are two cases to consider. The general fiber of \(f\) is of the 
form \(E^g\), so we again have \(\mathbf{MT}(X_b) = \mathbf{MT}(E)\). 
This is equal to \(\GL_2\) when \(E\) is not CM, and 
\(\mathrm{Res}_{K/\Q} \mathbb{G}_{m, K}\) when \(E\) is CM with endomorphism field \(K\). 
\qed
\end{example}

If one is to believe the statement of \Cref{general-mt-classification}, the 
behavior of the general fiber observed in this example is unusual; for a generically 
chosen Lagrangian fibration, we would expect that the very general fiber is simple. 
Of course, one cannot ``write down'' a general Lagrangian fibration. Still, 
we can certainly produce a Lagrangian fibration exhibiting a more typical behavior. 

\begin{example}[The Beauville--Mukai system]
Let \(S\) be a K3 surface, and \(H\) a primitive ample class on \(S\), with \(H^2 = 2g-2\).
Then \(|H| \simeq \bbP^g\) and the curves \(C \in |H|\) have genus \(g\). We can 
take the relative compactified Jacobians of the curves \(C \in |H|\) as follows:
let \(\mathbf{v} = (0, H, 1-g) \in H^0(S, \Z) \oplus H^2(S, \Z) \oplus H^4(S, \Z)\) be 
the given Mukai vector, and let \(X := M_H(\mathbf{v})\) be the moduli space of 
\(H\)-stable sheaves on \(S\) with Mukai vector \(\mathbf{v}\). 

To see that this is indeed a relative Jacobian, observe that \(X\) comes with a 
natural morphism \(f: X \to |H|\), taking each pure sheaf \(F \in X\) to the curve 
\(C \in |H|\) on which it is supported. 
In this case, \(F\) is the pushforward of some line bundle \(L\) on \(C\). 
By Riemann--Roch, we have \(\chi(F) = \chi(L) = \deg L + 1 - g\). But 
\(\chi(F) = 1-g\) by construction, so that \(\deg L = 0\). Therefore, we can 
identify the fiber \(X_{[C]} \subseteq M_H(\mathbf{v})\) with \(\mathrm{Pic}^0(C)\).

One might have expected that we would need to take \(X\) to be a moduli space of 
semi-stable sheaves for it to be compact. However, if we pick \(H\) to be sufficiently 
general, there are no strictly semi-stable sheaves in \(M_H(\mathbf{v})^\mathrm{ss}\).
Hence \(X\) is compact, and even projective. We claim that \(X\) is a hyper-K\"ahler 
manifold of dimension \(2g\). We will not prove this fact, but we will at least exhibit 
the symplectic pairing. Let \(F \in X\) be a sheaf. The associated tangent space can by 
identified with \(T_{[F]} X = \mathrm{Ext}^1(F, F)\). We let the symplectic pairing be 
the following composition:
\begin{equation*}
    \mathrm{Ext}^1(F, F) \times \mathrm{Ext}^1(F, F) \xrightarrow{\cup} 
    \mathrm{Ext}^2(F, F) \xrightarrow{\mathrm{Tr}} H^2(S, \mathcal{O}_S) \simeq \C.
\end{equation*}

Next we consider the Mumford--Tate group \(\mathbf{MT}(X_{[C]})\) of the very general 
fiber. Note that since \(H\) is ample, the smooth curves \(C \in |H|\) vary maximally, 
and hence so do their Jacobians. Thus the period map associated to \(f\) has 
\(g\)-dimensional image. In particular, the isotrivial cases seen in the previous examples
do not occur here. 

We claim that \(\mathbf{MT}(X_b) = \mathrm{GSp}_{2g}\). By Andr\'e's theorem, 
it is enough to show that the algebraic monodromy group of \(f\) is all of
\(\Sp_{2g, \Q}\). This follows from classical Picard--Lefschetz theory; one takes 
Leftschetz pencils \(\mathbb{P}^1 \subseteq |H|\) and observes that their algebraic 
monodromy groups are generated by Picard--Lefschetz cycles 
\(x \mapsto \left<x, \delta\right> \delta\). Since the nodal curves in \(|H|\) form 
an irreducible component of the discriminant of \(f\), the Picard--Lefschetz cycles 
form a single monodromy orbit. Our desired result follows from the following 
classical theorem.

\begin{theorem}[\cite{deligne80}, Lemma \(4.4.2^a\)]
    Let \((V, \left<,\right>)\) be a finite-dimensional complex symplectic space, 
    and \(M \leq \Sp(V)\) an algebraic subgroup. Suppose there is an \(M\)-orbit 
    \(R \subseteq V\) such that 
    \begin{enumerate}
        \item \(R\) spans \(V\),
        \item \(M\) is the smallest algebraic subgroup with the property that 
        for every \(\delta \in R\),
        \[\{x \mapsto x + \lambda\left<x, \delta\right>\delta \mid \lambda \in \C\} \subseteq M.\]
    \end{enumerate}
    Then \(M = \Sp(V)\).
\end{theorem}
\qed
\end{example}

So far, we have only produced Lagrangian fibrations coming from hyper-K\"ahler manifolds. 
However, if we allow \(X\) to be singular, we can exhibit more exotic Mumford--Tate groups.

\begin{proposition} \label{attainability}
    Let \(h, k \in \Z_{\geq 1}\) be any positive integers, and set \(g = hk\).
    There exists a primitive symplectic variety \(X\) of dimension \(2g\), 
    and a Lagrangian fibration \(f: X \to B\), such that the Mumford--Tate group 
    of the very general fiber 
    \[\mathbf{MT}(X_b) = (\mathbb{G}_m \times (\mathrm{Sp}_{2h})^k) / \pm 1. \]
\end{proposition}

In particular, the groups appearing in \Cref{mt-classification} are all attainable.

\begin{proof}
Let \(Y\) be a \(2h\)-dimensional hyper-K\"ahler manifold admitting a Lagrangian fibration 
\(\pi: Y \to S\) with \(\mathbf{MT}(Y_s) = \mathrm{GSp}_{2h}\). (For example, you can 
take \(\pi\) to be an appropriate Beauville--Mukai system.) Take the \(k\)-th symmetric 
power of \(Y\), \(X := Y^{(k)} = Y^k/\mathfrak{S}_k\). This is a primitive symplectic 
variety. Furthermore, the induced fibration 
\begin{equation*}
    f = \pi^{(k)}: X \to S^{(k)}
\end{equation*}
is Lagrangian. But by construction, the very general fiber of \(f\) has the form 
\(X_{(s_1, \dots, s_k)} = Y_{s_1} \times \cdots \times Y_{s_k}\), where each 
\(Y_{s_i}\) is distinct and Hodge--generic. This forces 
\(\mathbf{MT}(X_{(s_1, \cdots, s_k)}) = (\mathbb{G}_m \times (\mathrm{Sp}_{2h})^k) / \pm 1\), 
as needed.
\end{proof}

We conclude this section with the following observation.
The proof above show that each \(\mathbb{G}_m \cdot (\Sp_{2h})^k\) is attainable as 
the Mumford--Tate group of a {\it primitive symplectic} Lagrangian fibration. However, 
we can ask the following specialization of \Cref{attainability}: 
\begin{question}
    For which positive integers \(h\) and \(k\) does there exist a hyper-K\"ahler 
    manifold \(X\) of dimension \(2hk\) and a Lagrangian fibration \(f: X \to B\) with 
    \[\mathbf{MT}(X_b) = (\mathbb{G}_m \times (\mathrm{Sp}_{2h})^k) / \pm 1? \]
\end{question}

Notice that the argument in \Cref{attainability} no longer works in this setting. 
Indeed, when \(k \geq 2\), the symmetric power \(Y^{(k)}\) is not smooth.
Of course, we may try to find a crepant resolution for \(Y^{(k)}\), like
\(Y^{[k]} \to Y^{(k)}\) when \(h = 1\). However, when \(h \geq 2\), 
the cohomology of the base \(S^{(k)}\) is no longer equal to that of \(\bbP^n\). 
But by Matsushita \cite{matsushita16}, the base of a hyper-K\"ahler Lagrangian fibration 
must have the same intersection cohomology as projective space. Thus, 
if there were a resolution of \(Y^{(k)}\) by a hyper-K\"ahler, it would need to be 
indeterminate.

We expect that the only attainable values of \((h, k)\) are \((g, 1)\) and \((1, g)\), 
where \(g\) is any positive integer. These values were shown to be attainable in the 
examples above. However, we do not currently have an argument to rule out the 
intermediate cases.

\section{General singular fibers in a Lagrangian fibration}

For this section, we work with local analytic Lagrangian degenerations. To that end, 
let \(U \simeq \Delta^g\) be a polydisk, \(D \subseteq U\) an irreducible codimension 
one subvariety such that 
\(U^\circ := U \smallsetminus D \simeq \Delta^{g-1} \times \Delta^*\).

Let \((X, \sigma)\) be a holomorphic symplectic manifold of dimension \(2g\) equipped 
with a proper surjective morphism \(f: X \to U\), such that 
\(X_b\) is Lagrangian for all \(b \in U\), and \(X_b\) is smooth for \(b \in U^\circ\)
and singular when \(b \in D\). Let \(X^\circ\) be the preimage of \(U^\circ\) in \(X\). 
Set \(V = R^1 f_* \Z_{X^\circ}\) to be the variation of Hodge structures underlying \(f\).
Finally, let \(T \in \Sp_{2g}(\Z)\) be a generator for the monodromy group of \(f\) 
acting on the cohomology of a general fiber \(X_b\).

\begin{lemma}[\cite{bakker-schnell23}, Remark 6.8] \label{local-monodromy-eigenvalues}
    \(T\) has at most two non-trivial eigenvalues \(\zeta_i, \ol{\zeta}_i\), which 
    are the primitive \(i\)-th roots of unity for an integer \(i\) dividing six.
    If \(T\) has infinite order, \(i\) divides two.
\end{lemma}

\begin{proof}
Let \(N = \log T = N^{ss} + N^u\). Deligne's canonical extension extends a local frame 
\(v_1, \dots, v_{2g} \in H^0(U^\circ, \cV)\) to a frame 
\(\wt{v}_1, \dots, \wt{v}_{2g} \in H^0(U, \ol{\cV})\), 
\(\wt{v}_j = \exp(\frac{zN}{2\pi i}). v_j\), for \(z\) a local coordinate on the 
universal cover of \(U^\circ\). We can choose the \(v_j\) in such a way that \(N\) 
has the form 
\begin{equation*}
    N = \begin{pmatrix}
        D & S \\ 0 & D'
    \end{pmatrix},\qquad S = \begin{pmatrix}
        I_g & 0 \\ 0 & 0
    \end{pmatrix}
\end{equation*}
This frame has the advantage that it behaves well with respect to the flat connection:
if we set \(\varphi\) to be the local defining equation of \(D\), 
\begin{equation*}
    \nabla \wt{v}_k = \frac{N(\wt{v}_k)}{2\pi i} \frac{d\varphi}{\varphi}
    = \frac{\alpha_k}{2\pi i} \frac{d\varphi}{\varphi} \wt{v}_k + \frac{\beta_k}{2\pi i} \frac{d\varphi}{\varphi} \wt{v}_{g+k}.
\end{equation*}

The limit mixed Hodge structures of \(\cV\) define a locally split 
sub-bundle \(\left<\wt{v}_1, \dots, \wt{v}_g\right>|_D \subseteq \ol{\cV}|_D\).
After possibly shrinking the open subset \(U\), we may lift this to a Deligne frame 
\(h_1, \dots, h_g, h_{g+1}, \dots, h_{2g} \in H^0(U, \ol{\cV})\), 
\begin{equation*}
    F^1 \ol{\cV} = \left<h_1, \dots, h_g \right> \subseteq \ol{\cV},
\end{equation*}
agreeing with \(\left<\wt{v}_1, \dots, \wt{v}_g\right>\) on \(D\).

Now we have two exact sequences 
\begin{equation*}
\begin{tikzcd}
\left< \wt{v}_{g+1}, \dots, \wt{v}_{2g} \right> \arrow{r} \arrow[equal]{d} & \ol{\cV} \arrow{r} \arrow[equal]{d} & \ol{\cV} / \left< \wt{v}_{g+1}, \dots, \wt{v}_{2g} \right> \arrow{d}{g} \\
\left< \wt{v}_{g+1}, \dots, \wt{v}_{2g} \right> \arrow{r} & \ol{\cV}  \arrow{r} & \ol{\cV} / \left< h_{g+1}, \dots, h_{2g} \right>
\end{tikzcd}
\end{equation*}
Thus the two splittings differ by an element of 
\(\Hom(\ol{\cV} / \left< \wt{v}_{g+1}, \dots, \wt{v}_{2g} \right>, \left< \wt{v}_{g+1}, \dots, \wt{v}_{2g} \right>)\).
In other words, the isomorphism \(g\) has the form
\begin{equation*}
    h_j = \wt{v}_j + \sum_{k=g+1}^{2g} g_{jk} \wt{v}_k,\qquad j = 1,\dots, g.
\end{equation*}

Recall from \cite{bakker-schnell23} and our summary above the isomorphism 
\(\cT_U \simeq F^1 \ol{\cV} = \left<h_1,\dots, h_g\right>\).
This isomorphism is compatible with Lie bracket, so we have 
\begin{align*}
    [h_j, h_k] &= \nabla_{h_j} \wt{v}_k - \nabla_{h_k} \wt{v}_j 
    +  \sum_{\ell = g+1}^{2g} h_j(g_{k\ell})\wt{v}_\ell + g_{k \ell} \nabla_{h_j} \wt{v}_\ell   
    - h_k(g_{j\ell})\wt{v}_\ell - g_{j \ell} \nabla_{h_k} \wt{v}_\ell\\
    &= \frac{\alpha_k}{2\pi i} \frac{d\varphi(h_j)}{\varphi} \wt{v}_k - 
    \frac{\alpha_j}{2\pi i} \frac{d\varphi(h_k)}{\varphi} \wt{v}_j + 
    \frac{\beta_k}{2\pi i} \frac{d\varphi(h_j)}{\varphi} \wt{v}_{g+k} - 
    \frac{\beta_j}{2\pi i} \frac{d\varphi(h_k)}{\varphi} \wt{v}_{g+j} + \\
    &\qquad + \sum_{\ell = g+1}^{2g} \left( h_j(g_{k\ell}) - h_k(g_{j\ell}) + 
    \frac{\alpha_\ell}{2\pi i} \left(g_{k\ell}\frac{d\varphi(h_j)}{\varphi} - g_{j\ell}\frac{d\varphi(h_k)}{\varphi}\right) \right) \wt{v}_\ell \\
    &= \frac{\alpha_k}{2\pi i} \frac{d\varphi(h_j)}{\varphi} h_k - \frac{\alpha_j}{2\pi i} \frac{d\varphi(h_k)}{\varphi} h_j 
    + \frac{\beta_k}{2\pi i} \frac{d\varphi(h_j)}{\varphi} \wt{v}_{g+k} - 
    \frac{\beta_j}{2\pi i} \frac{d\varphi(h_j)}{\varphi} \wt{v}_{g+j} + \\
    &\qquad + \sum_{\ell = g+1}^{2g} \left( g_{j\ell} - g_{k\ell} + 
    h_j(g_{k\ell}) - h_k(g_{j\ell}) + \frac{\alpha_\ell}{2\pi i} 
    \left(g_{k\ell}\frac{d\varphi(h_j)}{\varphi} - g_{j\ell}\frac{d\varphi(h_k)}{\varphi}\right) \right) \wt{v}_\ell \\
    &= \frac{\alpha_k}{2\pi i} \frac{d\varphi(h_j)}{\varphi} h_k -
    \frac{\alpha_j}{2\pi i} \frac{d\varphi(h_k)}{\varphi} h_j.
\end{align*}
The last equality follows from the fact that the \(\wt{v}_\ell\) are not contained in 
\(\cT_U\) when \(\ell \geq g+1\).

Now, for \([h_j, h_k]\) to be holomorphic, we need either \(\alpha_k = 0\) or 
\(\varphi\) divides \(d\varphi(h_j)\). The latter is equivalent to \(h_j\) being tangent to \(D\).
But \(D\) is of codimension one, so at least one of the \(h_j\) is not tangent 
to \(D\), say \(h_\lambda\). This forces \(\alpha_k = 0\) for all \(k \neq \lambda\).

Notice additionally that \(g_{jk}\) all vanish on \(D\). Thus, when restricting 
to \(D\), the higher \(\wt{v}_\ell\) terms in the expression above yield 
\begin{align*}
    0 &= \frac{\beta_k}{2\pi i} \frac{d\varphi(h_j)}{\varphi} \wt{v}_{g+k} - 
    \frac{\beta_j}{2\pi i} \frac{d\varphi(h_k)}{\varphi} \wt{v}_{g+j} +
    \sum_{\ell = g+1}^{2g} \left(h_j(g_{k\ell}) - h_k(g_{j\ell}) \right) \wt{v}_\ell
\end{align*}
The sum on the right is holomorphic along \(D\), and \(\varphi\) does not divide
\(d\varphi(h_\lambda)\). Hence \(\beta_k = 0\) for all \(k \neq \lambda\). Hence we may 
lose no generality in writing that \(\lambda=1\); that is, the only possibly non-trivial
eigenvalues are \(\alpha_1\) and the only possibly non-semisimple eigenvector is 
\(\wt{v}_1\).

Observe that \(T\) is now an integer matrix with at most two non-trivial eigenvalues. 
By Grothendieck, all eigenvalues of \(T\) are roots of unity; if \(\zeta_i\) is a
primitive \(i\)-th root of unity, all other primitive \(i\)-th roots of unity 
are in its Galois orbit, and thus also occur as eigenvalues of \(T\). Thus 
if \(\phi\) is Euler's totient function, we must have \(\phi(i) \leq 2\). This 
forces \(i = 1,2,3,4,\) or \(6\). Furthermore, when \(T\) has infinite order,
there is only one non-trivial Jordan block, forcing \(\zeta_i = \ol{\zeta}_i\), 
or equivalently \(\zeta_i = \pm 1\).
\end{proof}

\section{Classification of Lagrangian Mumford--Tate groups}

\subsection{Existence of infinite local monodromy}

Now we work in the compact setting. Let \(X\) be a primitive symplectic variety of 
dimension \(2g\) which admits a maximally varying Lagrangian fibration \(f: X \to B\).
Take \(V\) to be the underlying variation of Hodge structures on \(B^\circ\).

\begin{proposition} \label{infinite-monodromy}
There exists an analytic neighborhood \(U \simeq \Delta^g \subseteq B\) meeting 
the discriminant \(D \subseteq B\) in a codimension one subvariety 
\(D_U := D \cap U \subseteq U\), such that the local monodromy of \(f\) around 
\(D_U\) is infinite.
\end{proposition}

Before proving this proposition, we remark on an important consequence. 
For one, recall from Hwang--Oguiso \cite{hwang-oguiso09} that \(D \subseteq B\) 
is either empty or codimension one. In the case where \(X\) is primitive symplectic, 
we can rule out the former case entirely.

\begin{corollary} \label{codim-one-discriminant}
    Let \(X\) be a primitive symplectic variety, and let \(f: X \to B\) be a Lagrangian
    fibration. Then the discriminant locus \(D \subseteq B\) of \(f\) has codimension one.
\end{corollary}

\begin{proof}
Since \(X\) is primitive symplectic, \(f\) is either maximally varying or isotrivial. 
By \Cref{infinite-monodromy}, the discriminant of \(f\) has a codimension one component 
whenever \(f\) varies maximally, so we may take \(f\) to be isotrivial. 
Now if \(D = \emptyset\), \(f\) is the trivial fibration and \(X \simeq X_b \times B\).
But by Kunneth, we must now have \(H^1(X, \C) \neq 0\), a contradiction.
\end{proof}

\begin{proof}[Proof of \Cref{infinite-monodromy}]

Our proof will use methods of \cite{grushevsky-mondello-manni-tsimerman25} 
to analyze the period image of \(V\).
Since \(V_\Q\) is irreducible, the period image lands in an 
indecomposable special subvariety of \(\cA_{g, \eta}\). By \cite[Theorem 4.5]{grushevsky-mondello-manni-tsimerman25}, 
the \(g\)-dimensional period image is not Hodge-generic inside of any special indecomposable 
non-compact subvariety. Thus we only need to show that the period image is not
contained in a compact special subvariety.

Let \(G \subseteq [\mathbf{MT}(V_\Q), \mathbf{MT}(V_\Q)]\) be the derived Mumford--Tate 
group of \(V\). By \cite[Corollary 8.5]{borel-tits65}, an algebraic group over a 
number field is anisotropic (contains no positive-dimensional split tori) if and only if 
it contains no non-trivial unipotent elements. By Schmid's nilpotent orbit theorem, 
for \(G\) this is equivalent to defining a compact special subvariety in \(\cA_{g, \eta}\).

Note that \(G\) is anisotropic if and only if its \(\Q\)-simple factors are, since 
a unipotent element in a factor induces a unipotent element in \(G\) and vice versa. Hence 
we may assume without loss of generality that \(G\) is a \(\Q\)-simple algebraic group.
By the classical theorem \cite[Corollary 2.7]{grushevsky-mondello-manni-tsimerman25}, 
we can write \(G = \mathrm{Res}_{K/\Q} H\), for \(K\) a totally real field and
\(H\) a geometrically simple \(K\)-algebraic group. But again, \(G\) is anisotropic 
if and only if \(H\) is anisotropic.

We will prove that \(H\) is not anisotropic. This will show that the image of the 
period map is not contained in a compact special subvariety. To do this, we 
will apply the results of \cite{scharlau85}. Let \(U \subseteq V_K\) the irreducible 
\(K\)-representation induced by the factor \(H \leq G_K\).

\begin{lemma}[\cite{grushevsky-mondello-manni-tsimerman25}, \S 5.1] \label{anisotropic-rep-classification}
    Suppose that \(H = \Aut(U, \eta)^\mathrm{ad}\) is non-compact at all real places. 
    Assume that one of the following conditions holds:
    \begin{enumerate}
    \item[(q)] \(U\) is a \(K\)-vector space, \(\eta\) is symmetric bilinear, 
        and \(\dim_K U \geq 5\). 
    \item[(uF)] \(U\) is an \(L\)-vector space for \(L/K\) a quadratic extension 
        with involution \(*\), \(\eta\) is a \(*\)-hermitian form, and \(\dim_L U \geq 3\).
    \item[(uH)] \(U\) is a right \(\mathbb{H}\)-module for \(\mathbb{H}/K\) a quaternion 
        division algebra with involution \(*\), \(\eta\) is a \(*\)-hermitian form, and 
        \(\dim_\mathbb{H} U \geq 2\).
    \item[(uD)] \(U\) is a right \(\mathbb{D}\)-module, for \(\mathbb{D}\) a non-commutative
        division algebra over \(L\) a quadratic extension of \(K\) with involution \(*\),
        \(\eta\) is a \(*\)-hermitian form, and \(\dim_\mathbb{D} V \geq 3\).
    \item[(skH)] \(U\) is a right \(\mathbb{H}\)-module, for \(\mathbb{H}\) a 
        quaternion division algebra over \(K\) with canonical involution \(*\), 
        \(\eta\) is a skew-hermitian form, and \(\dim_{\mathbb{H}} U \geq 4\).
    \end{enumerate}
    Then \(H\) is not anisotropic.
\end{lemma}

In our setting, \(V\) varies maximally, so \(H\) must be non-compact at all real places. 
Thus we may rule out all cases one-by-one. We roughly follow the analysis of 
\cite[\S 7.2]{grushevsky-mondello-manni-tsimerman25}. Satake classified all possible 
real forms of \(H\) up to isogeny, and we analyze each case. Let \(S\) denote the 
special subvariety associated to \(G\). 

For our analysis, we recall \cite[Table 3]{grushevsky-mondello-manni-tsimerman25},
providing a classification and numerics for the derived form of the group \(H\); 
see \Cref{table-of-algebraic-forms}. Now let us work by cases.

\begin{table}
\resizebox{1 \textwidth}{!}{
\begin{tabular}{ |c|c | c  | c |c|c|}
  \hline 
  &\textbf{Group $H^{\mathrm{der}}$} & \textbf{Non-compact}  &  
  \textbf{Irreducible}
  & \textbf{Non-self-dual/}\\
 &\textbf{over $K$} & \textbf{Hermitian} &  
 \textbf{complex linear} &   \textbf{/Symplectic/}\\
  &\textbf{(up to isogeny)} & \textbf{symmetric space}  & 
  \textbf{representation $U$} & \textbf{/Orthogonal}\\
  \hline\hline 
\begin{tabular}{c}
  (A$_1$)
\end{tabular} &  
  \begin{tabular}{c}
  $\textrm{Nm}_1(\mathbb{H})$\end{tabular} & 
  \begin{tabular}{c}
  $\mathrm{SU}_{1,1}/\mathrm{S}(\mathrm{U}_1\times\mathrm{U}_{1})$ \\ $\dim=1$
  \end{tabular} & 
  \begin{tabular}{c} (dual) standard $\C^2$\\ $\dim=2$\end{tabular} &
  Symp\\  
\hline 
  \begin{tabular}{c}
  (D$_4$)
\end{tabular} &  
  \begin{tabular}{c}
  ??? \\

  \end{tabular} & 
  \begin{tabular}{c}
  $\mathrm{SO}^*_8/\mathrm{U}_4=\mathrm{SO}_{6,2}/(\mathrm{SO}_6\times\mathrm{SO}_2)$
 \\ $\dim=6$
  \end{tabular} & 
  \begin{tabular}{c} standard \\ $\dim=8$\end{tabular} &
  Orth\\  
\hline 
\begin{tabular}{c}
  (I)$_{p,r\delta-p}$\\
  $r\delta\geq 3$
\end{tabular} &  
  \begin{tabular}{c}
  $\mathrm{SU}(\mathbb{D}^r,\eta)$\\ $\eta$ Hermitian\end{tabular} & 
  \begin{tabular}{c}
  $\mathrm{SU}_{p,r\delta-p}/\mathrm{S}(\mathrm{U}_p\times\mathrm{U}_{r\delta-p})$ \\ $\dim=p(r\delta-p)$
  \end{tabular} & 
  \begin{tabular}{c} (dual) standard $\C^{r\delta}$\\ $\dim=r\delta$\end{tabular} &
  NSD\\  
\hline 
\begin{tabular}{c}
  (I')$_{r\delta-1,1}^c$\\
  $r\delta\geq 4$\\
  $2\leq c\leq r\delta-2$\\
\end{tabular} & 
\begin{tabular}{c}
$\mathrm{SU}(\mathbb{D}^r,\eta)$\\ $\eta$ Hermitian\end{tabular} & 
\begin{tabular}{c}
$\mathrm{SU}_{r\delta-1,1}/\mathrm{S}(\mathrm{U}_{r\delta-1}\times\mathrm{U}_1)$ \\
$\dim=r\delta-1$
  \end{tabular}  & \begin{tabular}{c}$\bigwedge^c\C^{r\delta}$ \\ 
  $\dim=\binom{r\delta}{c}$\end{tabular}
& $\begin{cases} c\neq \frac{r\delta}{2} & \textrm{NSD}\\ c=\frac{r\delta}{2}\textrm{ even} & \textrm{Orth }\\  c=\frac{r\delta}{2}\textrm{ odd} & \textrm{Symp }  \end{cases}$\\  
\hline 
\begin{tabular}{c}
  (II)$_r$ \\ 
  $r\geq 2$ with $r\neq 4$ \end{tabular} &
\begin{tabular}{c}
$\mathrm{SU}(\mathbb{H}^r,\eta)$ \\
$\eta$ skew-Hermitian
\end{tabular}
& \begin{tabular}{c}
$\mathrm{SO}^*_{2r}/\mathrm{U}_r$\\
$\dim=\frac{r(r-1)}{2}$
\end{tabular}
&
\begin{tabular}{c}
standard\\
$\dim=2r$
\end{tabular} 
&
Orth
\\
\hline
\begin{tabular}{c}
  (III.1)$_r$ 
  \\ $r\geq 2$
\end{tabular} &
\begin{tabular}{c}
$\Sp(K^{2r},\eta)$ \\
$\eta$ symplectic
\end{tabular}
& \begin{tabular}{c}
$\Sp_{r}/\mathrm{U}_r$\\
$\dim=\frac{r(r+1)}{2}$
\end{tabular}
&
\begin{tabular}{c}
standard\\
$\dim=2r$
\end{tabular} &
Symp
\\
\hline \begin{tabular}{c}
  (III.2)$_r$ 
    \\ $r\geq 2$
\end{tabular} &
\begin{tabular}{c}
$\mathrm{SU}(\mathbb{H}^r,\eta)$ \\
$\eta$ Hermitian
\end{tabular}
& \begin{tabular}{c}
$\Sp_r/\mathrm{U}_r$\\
$\dim=\frac{r(r+1)}{2}$
\end{tabular}
&
\begin{tabular}{c}
standard\\
$\dim=2r$
\end{tabular} &
Symp
\\
\hline 
\begin{tabular}{c}
  (IV.1)$_{2p}$ 
  \\ $p\geq 3$, $p\neq 4$
\end{tabular} &
\begin{tabular}{c}
$\mathrm{SO}(K^{2p},\eta)$ \\
$\eta$ symmetric
\end{tabular}
& \begin{tabular}{c}
$\mathrm{SO}_{2p-2,2}/(\mathrm{SO}_{2p-2}\times\mathrm{SO}_2)$\\
$\dim=2p-2$
\end{tabular}
&
\begin{tabular}{c}
spin${^\pm(W)}$, $W$ isotropic\\
$\dim=2^{p-1}$
\end{tabular} 
&
$\begin{cases}p\equiv 2(4) & \textrm{Symp}\\ p\equiv 0(4) & \textrm{Orth}\\  p\equiv 1,3(4) &  \textrm{NSD}\end{cases}$
\\
\hline
\begin{tabular}{c}
  (IV.1)$_{2p+1}$ 
     \\ $p\geq 2$
\end{tabular} &
\begin{tabular}{c}
$\mathrm{SO}(K^{2p+1},\eta)$ \\
$\eta$ symmetric
\end{tabular}
& \begin{tabular}{c}
$\mathrm{SO}_{2p-1,2}/(\mathrm{SO}_{2p-1}\times\mathrm{SO}_2)$\\
$\dim=2p-1$
\end{tabular}
&
\begin{tabular}{c}
spin${(W)}$, $W$ isotropic\\
$\dim=2^p$
\end{tabular} 
&
$\begin{cases}p\equiv 0,3(4) &\textrm{Orth}\\ p\equiv 1,2(4) & \textrm{Symp}\end{cases}$
\\
\hline \begin{tabular}{c}
(IV.2)$_{2r}$ 
\\ $r\geq 3$ with $r\neq 4$
\end{tabular} &
\begin{tabular}{c}
$\mathrm{SU}(\mathbb{H}^r,\eta)$ \\
$\eta$ skew-Hermitian
\end{tabular}
& \begin{tabular}{c}
$\mathrm{SO}_{2r-2,2}/(\mathrm{SO}_{2r-2}\times\mathrm{SO}_2)$\\
$\dim=2r-2$
\end{tabular}
&
\begin{tabular}{c}
spin$^{\pm}{(W)}$, $W$ isotropic\\
$\dim=2^{r-1}$
\end{tabular} 
&
$\begin{cases}r\equiv 2(4) & \textrm{Symp}\\ r\equiv 0(4) & \textrm{Orth}\\  r\equiv 1,3(4) &  \textrm{NSD}\end{cases}$
\\
\hline \end{tabular}
}\\

\medskip

\caption{Satake's classification, 
\cite[Table 3]{grushevsky-mondello-manni-tsimerman25}} 
\label{table-of-algebraic-forms}
\end{table}

\begin{description}
    \item[\(H\) has type \(\mathrm{A}_1\)] We must have \(H = \mathrm{Nm}(\mathbb{H})\)
    and \(\dim U = 2\) for all \(i\). 
    Write \(\dim V = 2g = k \dim U = 2kh\). If \(k = 1\), \(H = \SL_{2, K}\) is not 
    anisotropic. Otherwise, \(g \leq \dim S = g/h\) yields a contradiction.
    \item[\(H\) has type \(\mathrm{D}_4\)] \(U_\C\) is a standard orthogonal representation of 
    dimension \(h = 8\). Since \(U\) is not symplectic, its dual must also occur 
    in the decomposition of \(V\), so \(\dim V = 16k\), while \(\dim S = 6\). 
    Thus \(g = 8k \leq 6\), a contradiction.
    \item[\(H\) has type \(\mathrm{I}\)] We are in (uF) and (uD) of \Cref{anisotropic-rep-classification},
    so \(H = \mathrm{SU}(\mathbb{D}^r, \eta)\) with \(r \leq 2\) and \(\dim_K U \leq 4\).
    Here \(\mathbb{D}\) has degree \(\delta \geq 2\) over a CM extension \(L/K\). 
    We may identify \(U_\C \simeq M_{\delta \times \delta}(\C)^{\oplus r}\) as 
    \(H(\C)\)-representations. Moreover, \(U\) is not self-dual, so its dual 
    occurs in the decomposition of \(V_\C\). In particular, 
    \begin{equation*}
        \dim V = 2g = 2k \dim U = 2 k \delta^2 r.
    \end{equation*}
    On the other hand, \(\dim S\) is maximized when the signature of \(H(\C)\) is 
    \((\frac{r\delta}{2}, \frac{r \delta}{2})\). Even in this case,
    \begin{equation*}
        g = k \delta^2 r \leq \dim S \leq k \cdot \frac{\delta r}{2} \cdot \frac{\delta r}{2} = 
        \frac{1}{4} k \delta^2 r^2,
    \end{equation*}
    so that \(r \geq 4\), a contradiction.
    \item[\(H\) has type \(\mathrm{I}'\)] \(H\) is the same as in type \(I\), and 
    \(U \simeq \bigwedge^c \C^{r \delta}\). Now \(\dim U = \binom{r \delta}{c}\), while 
    \(\dim S = r\delta - 1\). Since \(U\) is not self-dual, we cannot have 
    \(g \leq \dim S\) even when \(c = 1\).
    \item[\(H\) has type \(\mathrm{II}\)] Here \(H = \mathrm{SU}(\mathbb{H}^r, \eta)\) with 
    \(r \geq 2\) and \(\eta\) skew-hermitian. We are in the case (skH) of 
    \Cref{anisotropic-rep-classification}, so \(2 \leq \dim_\mathbb{H} U \leq 3\). 
    Note again that \(U\) is not self-dual, so that its dual also occurs in the 
    decomposition of \(V_\C\).
    Thus \(\dim V = 2g = 2k \dim U = 4kr\). On the other hand, \(\dim S = k\frac{r(r-1)}{2}\). 
    Thus \(g = 2kr \leq k\frac{r(r-1)}{2}\), so that \(4 \leq r-1\), a contradiction.
    \item[\(H\) has type \(\mathrm{III}\)] There are two possibilities here. Either 
    \(H = \Sp(K^{2r}, \eta)\), or \(H = \mathrm{SU}(\mathbb{H}^r, \eta')\). The 
    former case is always anisotropic, so we need not consider it. When \(r \geq 2\), 
    we are in case (uH) of \Cref{anisotropic-rep-classification}, so \(S\) is 
    non-compact. When \(r = 1\), \(H\) has type \(A_1\), which we have already disproved.
    \item[\(H\) has type \(\mathrm{IV}\)] \(H\) must be a form of \(\mathrm{SO}_n^\mathrm{ad}\), 
    for \(n \geq 5\). This puts us into the cases (q) and (skH) of \Cref{anisotropic-rep-classification}. 
    In either case, \(H\) is anisotropic.
\end{description}

\end{proof}

\subsection{Splitting the monodromy group also splits the representation}

Let \(S\) be the special Mumford--Tate group 
of \(V_\Q\), and let \(\Lambda = \mathrm{im}(\rho: \pi_1(B^\circ, b) \to \Sp_{2g, \Q})\). 
Define its Zariski closure \(\mathrm{Mon}(V_\Q) = \ol{\Lambda}\); it admits an exact sequence 
\begin{equation*}
    1 \to \mathrm{Mon}(V_\Q)^\circ \to \mathrm{Mon}(V_\Q) \to \Gamma \to 1,
\end{equation*}
where \(\Gamma\) is the component group of \(\mathrm{Mon}(V_\Q)\). By Andr\'e, 
\(\mathrm{Mon}(V_\Q)^\circ\) is a normal subgroup of \(S\). 

\begin{lemma}
    There exists a branched Galois cover of normal projective varieties \(\pi: B' \to B\) 
    such that \(\mathrm{Mon}(\pi^* V_\Q) = \mathrm{Mon}(V_\Q)^\circ\).
\end{lemma}

\begin{proof}
Construct \(\pi: B'^\circ \to B^\circ\) as follows: let \(\Lambda^\circ\) be the preimage 
of \(\mathrm{Mon}(V_\Q)^\circ\) under the natural inclusion. It is the preimage of a normal 
subgroup, and is therefore normal in \(\Lambda\). We have an exact sequence 
\begin{equation*}
    1 \to \Lambda^\circ \to \Lambda \to \Gamma \to 1.
\end{equation*}
Define \(B'^\circ\) to be the \'etale cover corresponding to the surjection 
\(\Lambda \to \Gamma\); that is, \(\Lambda^\circ = \pi_1(B'^\circ, b')\). Because
\(\Lambda^\circ\) is normal in \(\Lambda\), the cover is Galois.
Now we may define \(B'^\circ\) to be the normalization of \(B'^\circ\) in the function 
field \(K(B)\).
\end{proof}

To understand the monodromy of \(\pi^* V_\Q\), we will use Andr\'e's theorem, which 
implies that \(\mathrm{Mon}(V_\Q)^\circ \trianglelefteq S\). The remainder of this 
section will be dedicated to understanding the action of \(S\) on \(V_\Q\).


\begin{proposition}
    There exists an isomorphism of \(\C\)-Hodge structures 
    \(\pi^* V_\C = U_1 \oplus \cdots \oplus U_k\), where \(U_i\) has the compatible 
    structure of an irreducible \(\mathfrak{h}_i\)-representation and 
    \(\mathfrak{h}_j\) acts trivially on \(U_i\) for \(j \neq i\).
\end{proposition}

\begin{proof}
We have an isomorphism of complex Lie algebras 
\(\mathfrak{s}_\C = \mathfrak{h}_1 \oplus \cdots \oplus \mathfrak{h}_k\) coming from the 
isogeny decomposition. In particular, the decomposed Lie algebra acts on \(V_\C\) by 
\[d\rho: \mathfrak{h}_1 \oplus \cdots \oplus \mathfrak{h}_k \to \mathrm{Aut}(V_\C).\]
Let \(U\) be an irreducible factor of \(\pi^* V_\C\); we can write it as 
a tensor product \(U = W_1 \otimes \cdots \otimes W_k\), where \(W_i\) is an 
irreducible \(\mathfrak{h}_i\)-representation. 

Let \(\mu: \mathbb{G}_{m, \C} \to S(\C)\) be the Hodge cocharacter. Since 
\(V_\Q\) is weight 1, \(\mu(z)\) acts by \(z\) or \(0\). Consequently, 
\(d\mu: \C \to \mathfrak{s}_\C \to \mathrm{Aut}(V_\C)\) acts by matrices with
eigenvalues in \{0, 1\}. Denote by \(\mu_i: \C \to \mathfrak{h}_i \to \Aut(W_i)\)
the projection to the \(i\)-th factor. Suppose for \(z \in \C\) and \(w_i \in W_i\) 
we have \(d\mu(z).w_i = w_i\); then 
\begin{equation*}
    d\mu(z).(w_1 \otimes \cdots \otimes w_k) = 
    (\lambda_1 + \cdots + \lambda_k).(w_1 \otimes \cdots \otimes w_k),
\end{equation*}
where \(\lambda_j\) are the eigenvalues of \(d\mu(z)\) acting on \(W_j\). 
Since \(\lambda_i = 1\) and \(\lambda_1 + \cdots + \lambda_k \in \{0, 1\}\), 
we must have \(\lambda_j = 0\) for all \(j \neq i\). 

\begin{lemma}
    \(\mathfrak{h}_j\) acts trivially on \(W_j\) for \(j \neq i\). 
\end{lemma}

\begin{proof}
\(S\) is the smallest \(\Q\)-algebraic group whose real points contain the 
image of the Hodge cocharacter. As such, if \(X \in \mathfrak{s}_\C\) is a generator 
of the image of \(d\mu\), its projection to \(\mathfrak{h}_j\) must be
\(X_j \neq 0\) for all \(j=1,\dots, k\). Furthermore, \(\mathfrak{h}_j\) is 
simple for all \(j\), so that \([\mathfrak{h}_j, \mathfrak{h}_j] = \mathfrak{h}_j\).
In particular, \(X_j\) is traceless on \(W_j\).

Now if \(W_j\) is non-trivial, \(X_j\) must have at least one positive integer 
eigenvalue and at least one negative integer eigenvalue. Let \(\lambda_j \geq 1\) 
be the largest eigenvalue of \(X_j\) acting on \(W_j\). By the same argument, 
\(\lambda_i \geq 1\). This forces \(X\) to have an eigenvalue of at least two 
on \(U\), a contradiction. Thus the image of the Deligne torus \(X_j\) acts 
trivially on \(W_j\) for \(j \neq i\).

Note that we have not yet shown that \(\mathfrak{h}_j\) acts trivially, only that 
the representation \(W_j\) has level zero. Indeed, one may consider Mumford's 
famous example of a non-trivial family of Abelian 4-folds with Mumford--Tate group isogenous 
to \(\SL_2(\R) \times \mathrm{SU}(2)^2\) \cite{mumford69}. To rule out such examples, 
we will need to leverage some actual geometry native to Lagrangian fibrations.

Let \(D_1 \subseteq B\) be a component of the discriminant with infinite local 
monodromy; such a component exists by \Cref{infinite-monodromy}. Let \(T\) be a 
local monodromy operator for this component; its logarithm \(N\) belongs to 
\(\mathfrak{h}_1 \oplus \cdots \oplus \mathfrak{h}_k\). Since \(\mathfrak{h}_i\) 
acts with level one on \(W_i\), we may choose \(T\) such that the projection 
\(N_i \in \mathfrak{h}_i\) acts non-trivially. By \Cref{local-monodromy-eigenvalues}, 
\(N\) has rank one, so that all other projections \(N_j = 0\), for \(j \neq i\). 

Now observe that every other \(W_j\) is forced to be one-dimensional. Indeed, 
suppose \(\dim W_1 \geq 2\); then for every basis vector \(w_{1, \alpha} \in W_1\),
\begin{equation*}
    N. (w_{1, \alpha} \otimes \cdots \otimes w_i \otimes \cdots \otimes w_k) = 
    w_{1, \alpha} \otimes \cdots \otimes N_i . w_i \otimes \cdots \otimes w_k
\end{equation*}
provides a distinct linearly independent vector in \(\im(N)\), a contradiction.

Finally, since the action of \(\mathfrak{h}_j\) on \(W_j\) is forced to be given by 
characters, it factors through the Abelianization 
\(\mathfrak{h}_j/[\mathfrak{h}_j, \mathfrak{h}_j]\). But \(\mathfrak{h}_j\) is 
simple, so that this is the trivial group. Hence only \(\mathfrak{h}_i\) acts 
non-trivially on \(W_1 \otimes \cdots \otimes W_k\), as needed.
\end{proof}

Thus we may write \(\pi^* V_\C = (B_1 \otimes U_1) \oplus \cdots \oplus (B_k \otimes U_k)\); 
however, this is only a decomposition of infinitesimal representations. Next, we will show that 
the \(B_i \otimes U_i\) are in fact representations of \(S(\C)\). To that end, write 
\begin{equation*}
    \mathfrak{k}_i = \mathfrak{h}_1 \oplus \cdots \oplus \wh{\mathfrak{h}_i}
    \oplus \cdots \oplus \mathfrak{h}_k;
\end{equation*}
this allows us to describe \(B_i \otimes U_i\) as \(\mathrm{Ann}(\mathfrak{k}_i)\). We 
now claim that for all \(g \in S(\C)\), \(g B_i \otimes U_i = B_i \otimes U_i\). To see this, 
choose any \(X \in \mathfrak{k}_i\) and \(u_i \in B_i \otimes U_i\); by Lie theory, 
\begin{equation*}
    0 = g.(X.u_i) = g \frac{d}{dt}\Big|_{t=0} e^{tX} g^{-1} g . u_i = 
    (\mathrm{Ad}(g) X). (g . u_i).
\end{equation*}
But \(\mathrm{Ad}\) is a continuous action of \(S(\C)\) a connected group
on the discrete set of simple ideals of \(\mathfrak{s}_\C\). Hence the action does not 
permute the simple ideals, so that \(\mathrm{Ad}(g) \mathfrak{h}_i = \mathfrak{h}_i\), 
and analogously for \(\mathfrak{k}_i\). This implies that 
\(\mathrm{Ad}(g) X \in \mathfrak{k}_i\), so that 
\(g.u_i \in \mathrm{Ann}(\mathfrak{k}_i) = B_i \otimes U_i\), as needed.

Now we claim that \(\mathfrak{h}_1 = \cdots = \mathfrak{h}_k\). 
Indeed, by \cite{varvak25}, \(V_\C\) is irreducible, so by Clifford's theorem 
\cite{clifford37}, \(\Gamma\) permutes the \(B_i \otimes U_i\) factors
transitively; that is, \((B_i \otimes U_i)^\gamma = B_j \otimes U_j\), for all
\(\gamma \in \Gamma\). But recall that \(U_i = \mathrm{Ann}(\mathfrak{k}_i)\), so 
conjugation by \(\gamma\) induces an outer isomorphism 
\(\mathfrak{k}_i = \mathfrak{k}_j\), forcing \(\mathfrak{h}_i = \mathfrak{h}_j\), 
as needed.

Finally, we claim that \(B_i = \C\) for all \(i\). Indeed, \(V_\C\) varies maximally, 
so \(\pi^* V_\C\) also varies maximally, meaning that no power of a Hodge structure 
shows up in the very general fiber. This shows that 
\(\pi^* V_\C = U_1 \oplus \cdots \oplus U_k\), as needed.
\end{proof}

Notice that \(S(\C)\) is still merely isogenous to a power of a simple group.
We lose no generality in taking \(H^k\) to be the universal cover of \(S(\C)\);
thus \(\pi^* V_\C = U_1 \oplus \cdots \oplus U_k\) acquires the structure of a 
\(H^k\)-representation. Furthermore, the \(i\)-th factor \(H_i\) acts non-trivially 
on \(U_i\) and trivially on \(U_j\) for \(j \neq i\).

\begin{proposition}
    Let \(G = \mathbf{MT}(V_\Q)\) be the full Mumford--Tate group, and \(\wt{G}\) its 
    universal cover, \(\wt{G} = Z \times H_1 \times \cdots \times H_k\). 
    There exists an intermediate extension \(G' \twoheadrightarrow G\) 
    for which the Deligne torus \(h: \mathbb{S} \to G\) lifts to 
    \(h': \mathbb{S} \to G'\). Furthermore, we may write \(G'\) as a diagonal quotient
    \[G' = \wt{G}/(\underline{\zeta}) = 
        \frac{Z \times H_1 \times \cdots \times H_k}{(\zeta_0, \zeta_1, \dots, \zeta_k)}\] 
    for \(\zeta_0 \in Z\) and \(\zeta_i \in Z(H_i)\). 
\end{proposition}

\begin{proof}

In order for the Deligne torus to lift to an intermediate cover \(G' \to G\), 
we must have \(h': \mathbb{S} \to G'\) acting by a torus representation on \(U_i\).
In other words, \(dh' = X_0 + X_1 + \cdots + X_k\) must have integer eigenvalues.
This is the same as requiring \(\exp(2\pi i\cdot dh') = 1 \in G'\).
Thus we define \(\zeta_i = \exp(2\pi i X_i)\), \(i = 0, \dots, k\), and 
take \(\underline{\zeta} = (\zeta_0, \zeta_1, \dots, \zeta_k) \in \wt{G}(\C)\). 
By construction, the Deligne torus will lift to \(G' = \wt{G}/\underline{\zeta}\).
\end{proof}

To better understand why we generally cannot lift the Deligne torus to the universal 
cover \(\wt{G}\), consider the following example.

\begin{example}
Suppose \(G = \GL_2 \times_{\mathbb{G}_m} \GL_2\), acting on \(U_1 \oplus U_2\) by the 
standard representation. Then \(\wt{G}(\C) = \mathbb{G}_m(\C) \times \SL_2(\C)^2\). 
We can decompose the Hodge co-character as 
\begin{equation*}
    \mu = \begin{pmatrix}
        z & 0 & 0 & 0 \\
        0 & 1 & 0 & 0 \\
        0 & 0 & z & 0 \\
        0 & 0 & 0 & 1
    \end{pmatrix},\quad d\mu = \begin{pmatrix}
        1 & 0 & 0 & 0 \\
        0 & 0 & 0 & 0 \\
        0 & 0 & 1 & 0 \\
        0 & 0 & 0 & 0
    \end{pmatrix} = X_0 + X_1 + X_2.
\end{equation*}
Write the trace \(\mathrm{Tr}(X|_{U_1}) = \mathrm{Tr}((X_0 + X_1)|_{U_1}) 
= \mathrm{Tr}(X_0|_{U_1}) = 1\).
But \(X_0\) is central, so it is diagonal and acts by multiplication by \(1/2\).
Consequently,
\begin{equation*}
    X_0 = \begin{pmatrix}
        1/2 & 0 & 0 & 0 \\
        0 & 1/2 & 0 & 0 \\
        0 & 0 & 1/2 & 0 \\
        0 & 0 & 0 & 1/2
    \end{pmatrix},\  X_1 = \begin{pmatrix}
        1/2 & 0 & 0 & 0 \\
        0 & -1/2 & 0 & 0 \\
        0 & 0 & 0 & 0 \\
        0 & 0 & 0 & 0
    \end{pmatrix},\  X_2 = \begin{pmatrix}
        0 & 0 & 0 & 0 \\
        0 & 0 & 0 & 0 \\
        0 & 0 & 1/2 & 0 \\
        0 & 0 & 0 & -1/2
    \end{pmatrix}.
\end{equation*}
Now if we exponentiate these generators, we get 
\begin{equation*}
    \zeta_0 = -1 \in \mathbb{G}_m,\quad \zeta_1 = \zeta_2 = -I_2 \in Z(\SL_2).
\end{equation*}
Notice that \(G' = \wt{G}/(\ul{\zeta}) = G\), so the Deligne torus lifts no further
than \(G\). 
\qed
\end{example}


Now take the derived subgroup of \(G'\):
\begin{equation*}
    [G', G'] = [\wt{G}/N, \wt{G}/N] = \frac{[\wt{G}, \wt{G}] N}{N} = 
    \frac{[\wt{G}, \wt{G}]}{[\wt{G}, \wt{G}] \cap N},
\end{equation*}
where \(N = \left<\underline{\zeta}\right>\). Notice that all the \(\zeta_i\) are 
central, so that the intersection \([\wt{G}, \wt{G}] \cap N\) is trivial. Since 
\([\wt{G}, \wt{G}] = H_1 \times \cdots \times H_k\), we get the following corollary.


\begin{corollary}
    Let \(U_1' \oplus \cdots \oplus U_k'\) be the induced representation of \(G'\)
    on \(\pi^* V_\C = U_1 \oplus \cdots \oplus U_k\). 
    The derived subgroup \([G', G'] = H_1 \times \cdots \times H_k\) is a 
    power of a simple group, and acts on 
    \(U_1' \oplus \cdots \oplus U_k'\) with each \(H_i\) acting non-trivially
    on \(U_i'\) and trivially on \(U_j'\) for \(j \neq i\). 
\end{corollary}

\subsection{Proof of the main theorem}

Now we are ready to prove the main theorem. 

\begin{theorem}
    Let \(X\) be a primitive symplectic variety of dimension \(2g\), and let 
    \(f: X \to B\) be a Lagrangian fibration. Let \(G\) be the derived Mumford--Tate 
    group of \(f\). 
    \begin{equation*}
        G = \begin{cases}
            (\Sp_{2g/k})^k & : f \text{ varies maximally,} \\
            1 & : f \text{ is isotrivial with fiber isogenous to } E^g\\
            & \qquad \text{ for  } E \text{ a CM elliptic curve,}\\
            \SL_2 & : \text{ otherwise.}
        \end{cases}
    \end{equation*}
\end{theorem}

\begin{proof}
We first deal with the latter two cases. By \cite{bakker25}, if \(f\) does not vary 
maximally, it is isotrivial, so we can refer to \cite{kim-laza-martin24}. Namely, 
the fiber of \(f\) is isogenous to \(E^n\), for \(E\) an elliptic curve. 
Now \(G\) is \(\SL_2\) unless \(E\) is CM, in which case \(G\) is trivial. Thus 
we may assume that \(f\) varies maximally. 

As before, take \(\pi: B' \to B\) be the cover over which 
\(\pi^* V_\Q = U_1 \oplus \cdots \oplus U_k\) splits, and let \(\Gamma\) be the automorphism 
group of the cover, acting transitively on the \(U_i\). We claim that each of the 
\(U_i\) is an isotypic factor; that is, \(U_i \ncong U_j\) for \(i \neq j\). To see that this holds,
suppose that \(U_1 \cong U_2\). By \Cref{infinite-monodromy}, there exists a component 
\(D_1\) of the discriminant which has non-finite local monodromy. Without loss of 
generality, suppose that \(U_1\) degenerates along \(D_1\). Then \(U_2\) also degenerates 
along \(D_1\). But by \Cref{local-monodromy-eigenvalues}, every general degeneration is 
of nilpotent rank one, a contradiction.

Hence we must have \(U_i \ncong U_j\) for \(i \neq j\). We now claim that the 
splitting is defined over \(\Q\). Indeed, the splitting occurs over some number field 
\(K\), so we have a \(\Q\)-irreducible factor \(W \subseteq V_\Q\) which splits as 
\(W_K = U_1 \oplus \cdots \oplus U_\ell\). But now the Galois group \(\mathrm{Gal}(K/\Q)\) 
transitively permutes the \(U_1, \dots, U_\ell\). Now as above, suppose
\(U_1\) degenerates over \(D_1 \subseteq D\). The limit variation of mixed Hodge structures
of \(V\) over \(D_1\) is defined over \(\Q\), while that of \(U_1\) is not; thus 
any element of the Galois orbit of \(U_1\) must degenerate over \(D_1\) like \(U_1\).
But by \Cref{local-monodromy-eigenvalues}, at most one of the \(U_1\) can degenerate, 
forcing \(\ell = 1\), and therefore \(K = \Q\).

Now any irreducible factor of \(G\) is defined over \(\Q\). Set \(U = U_1\) and \(H = H_1\),
and consider the action of \(H\) on \(U\). The group \(H\) is geometrically simple, 
so that \(H(\C)\) is one of \(\SL_{2h}(\C)\), \(\Sp_{2h}(\C)\), or \(\mathrm{SO}_{2h}(\C)\).
However, we may rule out \(\SL_{2h}(\C)\) and \(\mathrm{SO}_{2h}(\C)\), since 
\(U\) comes from a family of Abelian varieties of dimension \(h\).

Recall that \(\Sp_{2h}(\C)\) has no non-trivial outer automorphisms. Thus 
\(H\) is an inner form of \(\Sp_{2h}\). Since \(U\) comes from a family of 
Abelian varieties (up to isogeny), \(H(\R) = \Sp_{2h}(\R)\); the only alternative 
would be \(\Sp(p, q)\), which is non-Hermitian when \(q \neq 0\), 
for which \(U\) is isotrivial. Now \(H\) must be of the form \(\Sp_r(\mathbf{D})\), for 
\(\mathbf{D}\) a division algebra with center \(\Q\), so that the fibers of \(U\) have 
the form \(U_b = \mathbf{D}^r\). Maximal parabolic subgroups of \(H\) are in bijection 
with totally isotropic \(\mathbf{D}\)-subspaces \(W \subseteq U_b\) via orbit--stabilizer. 
In particular, since \(\dim_\Q \mathbf D = d > 1\), the rank of the local monodromy logarithm of 
\(U\) around \(D_1\) is at least \(d\), contradicting \Cref{local-monodromy-eigenvalues}.
Therefore \(H = \Sp_{2h}\), and \(G = \Sp_{2h}^k\), concluding the proof. 
\end{proof}

Note that this allows us to compute the center of the special Mumford--Tate group. 
Since special endomorphisms are ruled out, the only endomorphisms which can 
commute with the special Mumford--Tate group are rational scalars. 
If we take \(c\) to be such a scalar, it satisfies 
\begin{equation*}
    c^2 \psi(v, w) = \psi(c I_{2g}.v, c I_{2g}.w) = \psi(v, w),
\end{equation*}
with respect to the polarization \(\psi\), forcing \(c^2 = 1\) and \(c = \pm 1\). 
However, since the special Mumford--Tate group is generated by the derived Mumford--Tate 
group and the identity component of the center, it is forced to be equal to \(G\).
This allows us to easily compute the actual Mumford--Tate group.

\begin{corollary} \label{mt-classification}
    Let \(X\) be a primitive symplectic variety of dimension \(2g\), and let 
    \(f: X \to B\) be a Lagrangian fibration. The Mumford--Tate group of 
    \(f\) is isogenous to 
    \begin{equation*}
        \mathbf{MT}(X_b) \sim \begin{cases}
            (\mathbb{G}_m \times \Sp_{2g/k}^k) / \{\pm 1\} & : f \text{ varies maximally, } 
            k \text{ divides g,}\\
            \mathrm{Res}_{K/\Q} \mathbb{G}_{m, K} & : 
            f \text{ is isotrivial with fiber isogenous to } E^g \text{ for } \\
            & \quad  E \text{ a CM elliptic curve with endomorphism field } K\\
            \GL_2 & : \text{ otherwise.}
        \end{cases}
    \end{equation*}
\end{corollary}

\subsection{An application to a question of Li--Tosatti} 

We will first need to make our notation a bit more precise. When we refer to 
``the'' special K\"ahler metric \(\omega_\mathrm{SK}\), we will simply fix one, 
since they are unique up to scaling. For any sufficiently small analytic neighborhood 
\(U \subseteq B^\circ\) there is a local coordinate system \(z_1, \dots, z_g\) 
such that the entries of the period matrix \(\mathbf{T}\) satisfy the 
Donagi--Markman symmetry condition \cite{donagi-markman94}:
\begin{equation} \label{eq-donagi-markman}
    \frac{\partial \tau_{ij}}{\partial z_k} = \frac{\partial \tau_{ik}}{\partial z_j}, 
    \quad 1 \leq i, j, k \leq g.
\end{equation}

\begin{proposition} \label{constant-diagonal-entry}
    Let \(X\) be primitive symplectic, and \(f: X \to B\) be a Lagrangian fibration 
    with period map \(\cP: B^\circ \to \cA_{g, \eta}\) and period matrix \(\mathbf{T}\), written as
    \begin{equation*}
        \cP(b) = \begin{pmatrix}
            \mathbf{T}(b) \\ I_g
        \end{pmatrix},\qquad \mathbf{T} = \begin{pmatrix}
            \tau_{11} & \cdots & \tau_{1g} \\ 
            \vdots & \ddots & \vdots \\ 
            \tau_{1g} & \cdots & \tau_{gg}
        \end{pmatrix}
    \end{equation*}
    Suppose that \(\mathbf{T}\) has a constant diagonal entry \(\tau_{ii}\)
    in some analytic neighborhood \(U \subseteq B^\circ\). Then \(f\) is isotrivial.
\end{proposition}

\begin{proof}
Without loss of generality, take \(i = 1\). By directly applying 
\eqref{eq-donagi-markman}, the condition that \(\tau_{11}\) is constant in all 
directions implies that
\begin{equation*}
    \frac{\partial \tau_{11}}{\partial z_1} = \frac{\partial \tau_{12}}{\partial z_1} = 
    \cdots = \frac{\partial \tau_{1k}}{\partial z_1} = 0.
\end{equation*}

Let \(V\) be the \(\Z\)-VHS underlying \(f\) on the locus of smooth fibers \(B^\circ\).
Choose a very general point \(p \in B^\circ\), and let \(C_p\) be a 
(locally defined) complex curve given by integrating \(\partial/\partial z_1\) 
through \(p\). We get a real Hodge substructure
\begin{equation}
    E = \begin{pmatrix}
        \tau_{11} \\ \tau_{21}\\ \vdots \\ \tau_{g1} \\ \\ | \\ e_1 \\ |
    \end{pmatrix} \oplus
    \begin{pmatrix}
        \overline{\tau}_{11} \\ \overline{\tau}_{21} \\ \vdots \\ \overline{\tau}_{g1} \\
        \\ | \\ e_1 \\ |
    \end{pmatrix} \subseteq V_{\R,p},
\end{equation}
which is preserved by moving along \(C_p\). In other words, \(E\) is contained in 
the fixed part of \(V_\R|_{C_p}\).

We can interpret \(C_p\) more moduli-theoretically. Recall that we have 
\(\cP: B^\circ \to \cA_g\); let \(\wt{\cP}: U \to \bbH_g\) be a local lift of \(\cP\) to 
the universal cover of \(\cA_{g, \eta}\). Then \(\wt{\cP}(U)\) meets the locus 
\(\{E\} \times \bbH_{g-1}\) along the curve \(\wt{\cP}(C_p)\). Observe that this is 
an atypical intersection: the image of \(\wt{\cP}\) is \(g\)-dimensional,
while the codimension of \(\{E\} \times \bbH_{g-1}\) is also
\(\binom{g+1}{2} - \binom{g}{2} = g\), so the expected dimension of the intersection 
is \(g - g = 0\). As a result,
we can apply the functional transcendence result of Bakker--Tsimerman 
\cite[Theorem 1.3]{bakker-tsimerman25}: because \(\wt{\cP}(U) \cap (\bbH_{g-1} \times \{E\})\) 
is atypical, \(\cP(C_p)\) is contained in a weakly special subvariety of \(\cA_g\).

However, recall that \(p\) was a very general point. Hence there is a 
weakly special subvariety \(Z \subsetneq \cA_{g, \eta}\) containing \(\cP(B^\circ)\). 
Such a subvariety by definition arises from a normal subgroup \(H\) of the 
Mumford--Tate group \(G = \mathbf{MT}(Z)\) of the very general point of \(Z\).
Now recall \Cref{mt-classification}; we want to rule out the case where \(G\) is
isogenous to \(\mathbb{G}_m \cdot (\Sp_{2h})^{g/h}\), for some \(h\) dividing \(g\). 

If \(h = g\), then we are immediately done, since \(G = \mathrm{GSp}_{2g}\) has 
only \(H = \mathbb{G}_m\) as a non-trivial normal subgroup, forcing \(\cP(C_p)\) to be a 
single point. This implies that \(f\) does not vary maximally, so it is 
isotrivial. 

On the other hand, if \(G\) is isogenous to \(\mathbb{G}_m \cdot (\Sp_{2h})^{g/h}\), 
the Mumford--Tate representation splits over \(Q\) as 
\(V_\Q = U_1 \oplus \cdots \oplus U_k\). Suppose that \(\mathbf{T}\) has a 
constant diagonal entry, say \(\mathbf{\tau}_{11}\). At least one of the \(U_i\) 
is such that \(e_1\) has a non-zero projection onto \(U_i\). In particular, 
this forces the period matrix \(\mathbf{T}_i\) of \(U_i\) to have a constant diagonal 
entry. But unlike \(V_\Q\), \(U_i\) is Hodge-generic; thus by the argument above, 
it must be isotrivial. This ensures that \(V_\Q\) does not vary maximally, 
so it too is isotrivial. 
\end{proof}

Notice that the proof of \Cref{constant-diagonal-entry} argues something else: 
the period matrix of \(V\) admits a constant entry if and only if the Mumford--Tate group 
of the very general fiber fails to be \(\mathrm{GSp}_{2g}\). 
In terms of \(\omega_{SK}\), we get the following consequence.

\begin{corollary} \label{omega_SK}
    Let \(X\) be primitive symplectic, let \(f: X \to B\) be a Lagrangian fibration, 
    and suppose that \(f\) is not isotrivial. Then \(\omega_\mathrm{SK}\) has positive 
    holomorphic sectional curvature. If we suppose further that 
    \(\mathbf{MT}(X_b) = \mathrm{GSp}_{2g}\), then \(\omega_\mathrm{SK}\) has 
    positive bisectional curvature. 
\end{corollary}

Lagrangian fibrations failing to induce \(\omega_{\mathrm{SK}}\) with positive 
bisectional curvature are 
readily realized, say by \Cref{example-K3n}. But in the next section, we will prove in
\Cref{general-mt-classification} that this can be an uncommon occurance: if \(X\) is a 
hyper-K\"ahler manifold with \(b_2(X) \geq 7\) such that \(f: X \to B\) is general 
among Lagrangian fibrations in that deformation class, then 
\(\mathbf{MT}(X_b) = \mathrm{GSp}_{2g}\).

\section{The Mumford--Tate group of a general deformation}

Recall that we had \(\pi: B' \to B\) such that \(\pi^* V_\C = U_1 \oplus \cdots \oplus U_k\),
a splitting we now know to occur over \(\Q\). Let 
\begin{equation*}
    1 \to \Gamma' \to \Gamma \to \Lambda \to 1
\end{equation*}
be the corresponding map of monodromy groups, where \(\Gamma\) is the monodromy image 
of \(V\), \(\Gamma'\) is the monodromy image of \(\pi^* V\), and \(\Lambda\) is the 
covering automorphisms of \(\pi\). 

Recall again that \(V_\C\) is irreducible as a \(\C\)-local system, by \cite{varvak25}. 
Since \(\pi^* V\) decomposes as above, the group \(\Lambda\) must transitively 
permute the irreducible factors \(U_i\). Since \(\Lambda\) is a quotient of \(\Gamma\), 
it is generated by local monodromy operators coming from some components 
\(D_i \subseteq D\).

\begin{lemma} \label{two-cycle-monodromy}
    There exists a component of the discriminant 
    \(D_1 \subseteq D\) whose associated local monodromy operator \(\lambda\) 
    acts by a two-cycle on the irreducible factors: for some \(i, j\) and all 
    \(\ell \neq i, j\),
    \[\lambda: U_i \to U_j,\ U_j \to U_i,\ U_\ell \xrightarrow{\sim} U_\ell, \quad 
    \lambda^2 = \id.\]
\end{lemma}

\begin{proof}
First notice that there is a natural map \(\Lambda \to \mathfrak{S}_k\), 
coming from the permutations of the factors \(U_1, \dots, U_k\) underlying
the action of \(\Lambda\). In fact, this proof will show that this map is surjective,
but it suffices that it surjects onto a subgroup \(\overline{\Lambda} \leq \mathfrak{S}_k\). 
Since \(\Lambda\) is the group of covering automorphisms of \(\pi\), there is a 
corresponding intermediate cover
\begin{equation*}
    B' \xrightarrow{\pi'} B'' \xrightarrow{p} B
\end{equation*}
where \(\ol{\Lambda} = \Aut(p)\). Take \(D_1\) to be a component whose local 
monodromy operator acts by an element of \(\ol\Lambda\); then \(\lambda\) 
acts on any \(U_{i}\) by a \(c_i\)-cycle; that is,
\begin{equation} \label{eq-permutation-of-factors}
    \lambda: U_{i} \to U_{i_2} \to \cdots \to U_{i_{c_i}} \to U_{i}, \quad 
    \lambda^{c_i}: U_i \xrightarrow{\sim} U_i.
\end{equation}

Now notice that \(\ol\Lambda\) acts on the limit mixed Hodge structure over a general 
point of \(D_1\). In other words, for \(b \in D_1\) a general point, 
we can find a local frame
\(v_{i_1}, \dots, v_{i_{c_i}} \in (F^1 \ol{\cV}|_{D_1})\) for which
\begin{equation*}
    \lambda: v_{i_1} \mapsto v_{i_2} \mapsto \cdots \mapsto v_{i_{c_i}} \mapsto v_{i_1}.
\end{equation*}

Fix \(\zeta\) a primitive \(c_i\)-th root of unity, and consider the action of \(\lambda\) 
on \(v_{i_1} + \zeta v_{i_2} + \cdots + \zeta^{c_i-1} v_{i_{c_i}}\); this is an eigenvector
in \(F^1 \ol\cV|_{D_1}\) with eigenvalue \(\ol\zeta\). On the other hand, 
\(v_{i_1} + \ol\zeta v_{i_2} + \cdots + \ol\zeta^{c_i-1} v_{i_{c_i}}\) is another eigenvector 
living in \(F^1 \ol\cV|_{D_1}\), with eigenvalue \(\zeta\). By \Cref{local-monodromy-eigenvalues},
we cannot have two distinct eigenvalues in the \(F^1\)-piece of the limit MHS, 
so we must have \(\zeta = \ol\zeta = -1\). In particular, \(c_i = 2\). 

Finally, the argument that \(\lambda\) fixes all but two of the \(U_i\) is analogous: 
if not, by the construction above, we could produce two distinct monodromy eigenvectors 
in the \(F^1\) piece of the limit MHS. This would once again contradict 
\Cref{local-monodromy-eigenvalues}.
\end{proof}

We have shown that if the very general fiber of \(f: X \to B\) is not Hodge-generic, 
then it has general singular fibers whose associated local monodromy group is \(\Z/2\Z\).
Recall from the classification of singular fibers of Hwang--Oguiso \cite{hwang-oguiso09} 
that such a fiber is of Kodaira type \(\mathrm{I}_0^*\). Now recall the following result 
of Lehn:
\begin{proposition}[\cite{lehn2016}, Theorem 5.7] \label{lehn-general-singular-fibers}
    Let \(X\) be a hyper-K\"ahler manifold and \(f: X \to B\) a Lagrangian fibration.
    Take \(X'\) a general deformation of \(X\) preserving the fibration, \(f': X' \to B'\). 
    All general singular fibers of \(f'\) are of Kodaira type \(\mathrm{I}_n\), 
    \(\mathrm{II}\), \(\mathrm{III}\), or \(\mathrm{IV}\).
\end{proposition}

Coupled with our technical result \Cref{two-cycle-monodromy}, this directly proves 
the following theorem.

\begin{theorem} \label{general-mt-classification}
    Let \(\mathcal{X}\) be a deformation class of hyper-K\"ahler manifolds satisfying 
    \begin{enumerate}
        \item there exists \(X \in \mathcal{X}\) a hyper-K\"ahler manifold admitting a 
        Lagrangian fibration,
        \item the second Betti number of any member \(X \in \mathcal{X}\) is at least 7.
    \end{enumerate}
    Choose \(X \in \mathcal{X}\) general among those admitting Lagrangian fibrations,
    \(f: X \to B\). Then the Mumford--Tate group of the very general fiber \(X_b\) 
    is \(\mathrm{GSp}_{2g}\).
\end{theorem}

\begin{proof}
By work of Kim--Laza--Martin \cite[Theorem 6.2]{kim-laza-martin24}, \(f\) is not 
isotrivial. Suppose that \(X_b\) were not Hodge-generic; then \Cref{two-cycle-monodromy}
applies, and \(f\) has a general singular fiber of type \(\mathrm{I}_0^*\), 
contradicting \Cref{lehn-general-singular-fibers}.
\end{proof}

\begin{remark}
Note that the statements of \cite[Theorem 5.7]{lehn2016} and 
\cite[Theorem 6.2]{kim-laza-martin24} in their respective papers are somewhat weaker 
than what is used here, since they only guarantee the existence of an appropriate
deformation. However, the proofs of both theorems show that these results hold for the 
general deformation.
\end{remark}

\bibliographystyle{alphaurl}
\bibliography{refs}

\end{document}